\documentclass[12pt]{amsart}
\usepackage{graphicx}
\usepackage[centertags]{amsmath}
\usepackage{amsfonts}
\usepackage{amssymb}
\usepackage{amsthm}
\newtheorem{thm}{Theorem}

\newtheorem{lem}{Lemma}[section]
\newtheorem{cor}{Corollary}[section]

\newtheorem{prop}[lem]{Proposition}
\theoremstyle{definition}
\newtheorem{defn}[lem]{Definition}
\theoremstyle{remark}
\newtheorem{rem}{Remark}[section]
\numberwithin{equation}{section}

\newcommand{\norm}[1]{\left\Vert#1\right\Vert}

\newcommand{\set}[1]{\left\{#1\right\}}

\newcommand{\pr}[2]{\langle  #1,#2\rangle}

\newcommand{\G}{\Gamma}

\newcommand{\calA}{\mathcal{A}}
\newcommand{\calC}{\mathcal{C}}

\newcommand{\calF}{\mathcal{F}}

\newcommand{\calN}{\mathcal{N}}
\newcommand{\calO}{\mathcal{O}}
\newcommand{\calP}{\mathcal{P}}
\newcommand{\calR}{\mathcal{R}}

\newcommand{\bbZ}{\mathbb{Z}}

\newcommand{\bbQ}{\mathbb{Q}}
\newcommand{\bbR}{\mathbb R}
\newcommand{\bbC}{\mathbb C}

\newcommand{\bbH}{\mathbb{H}}

\newcommand{\Tr}{ \mbox{Tr}}

\newcommand{\SL}{ \mathrm{SL}}
\newcommand{\SO}{ \mathrm{SO}}

\newcommand{\PSL}{ \mathrm{PSL}}

\newcommand{\SU}{ \mathrm{SU}}
\newcommand{\GL}{ \mbox{GL}}

\newcommand{\Mat}{\mathrm{Mat}}

\newcommand{\ind}{\mathrm{Ind}}

\newcommand{\vol}{\mathrm{vol}}

\newcommand{\lap}{\triangle}
\newcommand{\bs}{\backslash}

\newcommand{\reg}{\mathrm{Reg}}
\newcommand{\res}{\mathrm{Res}}
\newcommand{\f}{\mathfrak{f}}
\newcommand{\frakD}{\mathfrak{D}}

\newcommand{\g}{\mathfrak{g}}

\begin{document}
\title[Strong Spectral gaps]
{Strong Spectral Gaps for Compact Quotients of Products of $\PSL(2,\bbR)$}%
\author{Dubi Kelmer and Peter Sarnak}%
\address{Department of Mathematics, University of Chicago,  5734 S. University
Avenue Chicago, Illinois 60637}
\email{kelmerdu@math.uchicago.edu}
\address{Department of Mathematics, Princeton University, Fine
Hall,Washington Road, Princeton, NJ 08544 and Institute for Advanced Study ,Princeton NJ 08540}
\email{sarnak@Math.Princeton.edu}

\thanks{The first author was partially supported by the NSF grant DMS-0635607, and the second author by the NSF grant DMS-0758299.}%

\date{\today}%
\dedicatory{}%
\commby{}%

\begin{abstract}
The existence of a strong spectral gap for quotients $\Gamma\bs G$ of
noncompact connected semisimple Lie groups is crucial in many
applications. For congruence lattices there are uniform and very
good bounds for the spectral gap coming from the known bounds
towards the Ramanujan-Selberg Conjectures. If $G$ has no compact
factors then for general lattices a strong spectral gap can still be
established, however, there is no uniformity and no effective bounds
are known. This note is concerned with the strong spectral gap for an
irreducible co-compact lattice $\Gamma$ in $G=\PSL(2,\bbR)^d$ for
$d\geq 2$ which is the simplest and most basic case where the
congruence subgroup property is not known. The method used here
gives effective bounds for the spectral gap in this setting.
\end{abstract}

\maketitle
\section*{introduction}
This note is concerned with the strong spectral gap property for an
irreducible co-compact lattice $\Gamma$ in $G=\PSL(2,\bbR)^d,\;
d\geq 2$. Before stating our main result we review in some detail
what is known about such spectral gaps more generally. Let $G$ be a
noncompact connected semisimple Lie group with finite center and let
$\Gamma$ be a lattice in $G$. For $\pi$ an irreducible unitary
representation of $G$ on a Hilbert space $H$, we let $p(\pi)$ be the
infimum of all $p$ such that there is a dense set of vectors $v\in
H$ with $\pr{\pi(g)v}{v}$ in $L^p(G)$. Thus if $\pi$ is finite
dimensional $p(\pi)=\infty$, while $\pi$ is tempered if and only if
$p(\pi)=2$. In general $p(\pi)$ can be computed from the Langlands
parameters of $\pi$ and for many purposes it is a suitable measure
of the non-temperedness of $\pi$ (if $p(\pi)>2$). The regular
representation, $f(x)\mapsto f(xg)$, of $G$ on $L^2(\Gamma\bs G)$ is
unitary and if $\Gamma\bs G$ is compact it decomposes into a
discrete direct sum of irreducibles while if $\Gamma\bs G$ is
non-compact the decomposition involves also continuous integrals via
Eisenstein series. In any case, let $E$ denote the exceptional
exponent set defined by
\begin{equation}
E(\Gamma\bs G)=\left\lbrace\begin{array}{cc}
& p(\pi)> 2, \mbox{ and $\pi$ is an infinite}\\
p(\pi): &  \mbox{dimensional irreducible representation}\\
& \mbox{of $G$  occurring weakly in } L^2(\Gamma\bs G)\end{array}\right\rbrace
\end{equation}
If $E(\Gamma\bs G)$ is empty set $p(\G\bs G)=2$ and otherwise let
\begin{equation}
p(\Gamma\bs G)=\sup E(\Gamma\bs G)
\end{equation}
We say that $\Gamma\bs G$ has a strong spectral gap if $p(\Gamma\bs
G)<\infty$. The existence of such a gap is critical in many
applications. In the case that $\Gamma$ is a congruence group (this
is the automorphic case discussed below) the set $E$ and the precise
value of $p$ are closely connected to the generalized Ramanujan
conjectures \cite{Clozel07,Sarnak05}. In ergodic theoretic
applications $p(\Gamma\bs G)$ controls the precise mixing rate of
the action of noncompact subgroups of $G$ on $\G\bs G$
\cite{Moore66,HoweMoore79,Schmidt81}.  In questions of local rigidity of related
actions the spectral gap controls the ``small divisors'' in the
linearized co-cycle equations \cite{DamjanovicKatok07} and it plays an important
role in the study of the cohomology of $\Gamma$ \cite{BergeronClozel05,BorelWallach80}.

The congruence case is defined as follows. Let $H$ be a semisimple
linear algebraic group defined over a number field $F$ and let
$S_\infty$ denote the set of archimedean places of $F$. The group
$G=\prod_{\nu\in S_\infty} H(F_\nu)$, where $F_\nu$ is the
completion of $F$ at $\nu$, and $\Gamma$ is a congruence subgroup of
$H(F)$ embedded into $G$ diagonally. After \cite{BurgerSarnak91} and
\cite{Clozel03}, which establish bounds towards Ramanujan
Conjectures in general, one knows that $p(\Gamma\bs G)$ is finite in
these cases. In fact the methods used there yield explicit, and in
many cases quite sharp, bounds for $p(\Gamma\bs G)$ which depend
only on $H$ and not on $\Gamma$. The latter is crucial in many
number theoretical as well as group theoretic applications
\cite{Lubotzky94,Sarnak05}. Arthur's conjectures
\cite{Arthur02,Clozel07} for the discrete spectrum for such spaces
$\G\bs G$ imply strong restrictions on the non-tempered $\pi$'s that
can occur. Specifically they must correspond to local Arthur
parameters which gives a ``purity'' property \cite[Chapter
6]{BergeronClozel05} and which in turn restricts the set $E(\Gamma\bs
G)$. In particular the set $E(\Gamma\bs G)$ should be finite, though
the set of non-tempered $\pi$'s occurring in $L^2(\G\bs G)$ can
certainly be infinite.

Two basic congruence examples are $(i)$ $G=\SL(2,\bbR)$ and $\G$ a
congruence subgroup of $\SL(2,\bbZ)$ in which case Selbergs
eigenvalue Conjecture \cite{Selberg65} is equivalent to $E(\G\bs
G)=\emptyset$, while it is known that $E(\G\bs G)$ is finite and is
contained in $(2,\frac{64}{25}]$ \cite{KimSarnak03}. $(ii)$
$G=\SL(3,\bbR)$ and $\G$ a congruence subgroup of $\SL(3,\bbZ)$ in
which case Langlands Conjectures for automorphic cuspidal
representations on $\GL_n$ \cite{Langlands01} imply that $E(\G\bs
G)=\{4\}$. The exceptional exponent comes from the unitary
Eisenstein series for the maximal parabolic subgroup. From
\cite{KimSarnak03} it follows that $E(\G\bs G)\subset \{4\}\cup
(2,\frac{28}{9}]$ but here  $E$ is not known to be finite.

Returning to the general lattice $\Gamma$, we may, without any
serious loss of generality take $G$ to be the direct product
$G_1\times G_2\times\cdots\times G_n$ of simple Lie groups with
trivial center and assume that $\G$ is irreducible. By the latter we
mean that if $G=G_1\times G_2\cdots\times G_r\times G_c$ with $G_j$
noncompact for $j=1,\ldots, r$ and $G_c$ compact, that the
projection of $\Gamma$ onto the compact factor is dense and if $r>1$
so are the projections of $\Gamma$ on each $G_j,\;j=1,\ldots, r$.
This implies that the only $G_j$ invariant vector in $L^2(\G\bs G)$
is the constant function.

If $G$ has no compact factors then $L^2(\Gamma\bs G)$ has a strong
spectral gap. To see this consider separately the case that $r=1$
and $r>1$. If $r=1$ and the rank of $G$ is $1$, then the spectral
gap follows directly from the discreteness of the spectrum of the
Laplacian below the (possible) continuous spectrum on $\G\bs G/K$
where $K$ is a maximal compact subgroup of $G$ (in the cases
$G=\SO(n,1)$ and $\SU(n,1)$, $p(\G\bs G)$ can be arbitrarily large
as one varies over $\Gamma$, this is shown for $\SO(2,1)$ in
\cite{Selberg65} by starting with a $\G$ with $H_1(\G)$ infinite and
this can be done in the same way for these $G$'s). If the rank of G
is at least 2 then G has property T and $p(\Gamma\bs G)$ is less
than or equal to $p(G)$ which is finite \cite{Cowling79}. The
optimal exponent $p(G)$ associated with such a $G$ has been
determined in many cases including classical groups
\cite{Howe89,Li95} and some exceptional groups \cite{LokeSavin06},
while explicit and strong upper bounds for $p(G)$ are given for split
exceptional groups in \cite{LiZhu96} and in
complete generality in \cite{Oh02}. If $r\geq 2$, we need to use
more machinery to deduce the spectral gap. Firstly by Margulis
\cite[Capter IX]{Margulis91}, $\Gamma$ is arithmetic and hence is
commensurable with a congruence lattice of the type discussed in the
previous paragraph, for which we have a strong spectral gap. This
coupled with the Lemma of Furman-Shalom, Kleinbock-Margulis (see
\cite[page 462]{KleinbockMargulis99}) allows one to pass from the
congruence group to $\Gamma$ and to conclude that $p(\G\bs
G)<\infty$. Note that any $\pi$ occurring in $L^2(\Gamma\bs G)$ is
of the form $\pi\cong \pi_1\otimes \pi_2\cdots\otimes \pi_r$, with
$\pi_j$ an irreducible representation of $G_j$ and that
$p(\pi)=\max_j p(\pi_j)$ (it is this $\max p(\pi_j)$ that is the
issue and which makes the problem difficult, if we used $\min
p(\pi_j)$ we could proceed as in the case $r=1$). In applying the
above lemma one loses all information in terms of specifying
$p(\G\bs G)$. While the analysis can be made effective in principle,
doing so would be unwieldy and the bound would anyway depend very
poorly on $\Gamma$. For arithmetic applications the latter is a
serious defect. We remark that in this case that $G$ has no compact
factors we don't know if $E(\G\bs G)$ is necessarily finite.

When $G$ has a compact factor the situation is apparently more
difficult. In the first place it is not known in general that $\G\bs
G$ has a strong spectral gap. The most problematic case is the
simplest one, that is $G=SL(2,\bbR)\times \SU(2)$. The suggestion
(2) in \cite[page 57]{gamburdJakobsonSarnak99} is equivalent to the
existence of a strong spectral gap for any irreducible $\Gamma$ in
such a $G$. In \cite{gamburdJakobsonSarnak99} this spectral gap is
proved for many $\Gamma$'s and this has been extended (using novel
methods from additive combinatorics) in \cite{BourgainGamburd08} to
include any $\Gamma$ whose projection on $\SU(2)$ consists of
matrices with algebraic numbers as entries. However in this case of
$G$ having compact factors, the set $E(\G\bs G)$ can be far more
complicated. Borrowing a technique in
\cite{LubotzkyPhillipsSarnak87} we show the following
\begin{thm}\label{t:generic}
There is an irreducible $\Gamma$ in $G=\SL(2,\bbR)\times \SU(2)$ for
which $E(\G\bs G)$ is infinite. In fact this is so for the generic
$\G$.
\end{thm}

Next, we turn to the simplest and most basic case for which an
effective spectral gap is not known, that is for $\Gamma$ an
irreducible co-compact lattice in $G=\PSL(2,\bbR)^d,\;d\geq 2$. Such
a $\Gamma$ is arithmetic and from the classification of such groups
\cite{Weil60} we have that $\Gamma$ is commensurable with the unit
group in a suitable division algebra (see section \ref{s:QuatAlg}).
Serre conjectures that the congruence subgroup property holds for
such groups, this being the most elementary and fundamental case for
which the congruence subgroup problem is open (see \cite[Chapter
7]{LubotzkyAlexander03}). If true, this coupled with the
Jacquet-Langlands correspondence \cite{JacquetLanglands70} yields
that $E(\G\bs G)$ is empty if the Ramanujan-Selberg conjecture
\cite{Selberg65} is true, and that $E(\Gamma\bs G)\subseteq
(2,\frac{18}{7}]$ using \cite{KimShahidi02}.

We can now formulate our main result. We will work in a slightly
more general setting allowing $\Gamma$ to act via a unitary
representation. Let $\rho:\Gamma\to U(N)$ be an $N$-dimensional
unitary representation of $\Gamma$. Let $L^2(\Gamma\bs G,\rho)$
denote the space of functions from $G$ to $\bbC^N$ satisfying
\begin{eqnarray}
 f(\gamma g)=\rho(\gamma)f(g),\\
\nonumber \int_{\G\bs G}|f(g)|^2dg<\infty.
\end{eqnarray}
The regular representation $f(x)\mapsto f(xg)$ of $G$ on
$L^2(\Gamma\bs G,\rho)$ decomposes discretely as
\begin{equation}
L^2(\G\bs G,\rho)\cong \bigoplus_{k=0}^\infty \pi_k(\rho),
\end{equation}
with $\pi_k(\rho)$ irreducible representations of $G$.

\begin{thm}\label{t:gap}
Let $\Gamma\subseteq\PSL(2,\bbR)^d$ be an irreducible co-compact
lattice and $\rho$ and $\pi_k(\rho)$ be as above. Then for any
$\alpha>0$, $p(\pi_k(\rho))<6+\alpha$ except for at most a finite
number of $k$'s. In particular
\[E(\Gamma\bs G)\cap [6+\alpha,\infty)|<\infty.\]
\end{thm}

\begin{rem}\label{r:1}
From the arithmeticity of $\Gamma$ ($n>2$), we know that it is
commensurable to a lattice $\Gamma_\calA$ derived from a quaternion
algebra. We can thus assume (replacing $\Gamma$ by
$\Gamma\cap\Gamma_\calA$ if necessary) that $\Gamma\subseteq
\Gamma_\calA$ is a finite index subgroup. Moreover, since we can
also replace the representation $\rho$ by the induced representation
$\ind_\Gamma^{\Gamma_\calA}\rho$, it is sufficient to prove the
theorem only in the case where $\Gamma=\Gamma_\calA$.
\end{rem}

\begin{rem}\label{r:2}
The theorem implies that $p(\Gamma,\rho)<\infty$ and much more. The
proof yields effective bounds (polynomial in $\dim\rho$) both for
the number of exceptions as well as bounds for $p(\pi_k)$ for these
exceptions. For some applications  the finite number of exceptions
enter as secondary terms in rates of equidistribution and are
harmless, so that the theorem is effectively asserting that
$p(\Gamma,\rho)\leq 6$.
\end{rem}

\begin{rem}
The proof of the theorem is based on the Selberg trace formula
\cite{Efrat87,Hejhal76} and counting arguments involving relative
quadratic extensions of $L$ (the field of definition of the
corresponding quaternion algebra) as in \cite{Shimizu}. One can
probably combine the analysis here with that in \cite{SarnakXue91}
(see also \cite{Huxley84}) to show that for any fixed $\Gamma$ as
above and any $\Lambda$ a congruence subgroup of $\G$ (i.e., the
intersection of $\Gamma$ with a congruence subgroup of the unit
group of the quaternion algebra) that the exceptional  $\pi_k$'s for
$L^2(\Lambda\bs G)$ with $p(\pi_k)>6+\alpha$, consists only of the
finite number of $\pi_k$'s that are there from $\G$ (i.e., no new
exceptional $\pi$'s appear in passing from $\Gamma$ to $\Lambda$).
We have not carried this out and doing so would be of interest since
for most applications this uniform spectral gap is a good substitute
for the Ramanujan Conjectures.
\end{rem}

We apply the theorem to the Selberg Zeta function in this setting.
For simplicity we take $d=2$ and $\Gamma$ torsion free. Each $1\neq
\gamma\in\Gamma$ is of the form $(\gamma_1,\gamma_2)$ with
$\gamma_j\in\PSL(2,\bbR)$ and $\gamma_j\neq 1$. We call $\gamma$
mixed if $\gamma_1$ is hyperbolic and $\gamma_2$ is elliptic. That
is $\gamma_1$ is conjugate to $\begin{pmatrix} N(\gamma)^{1/2} & 0\\
0 & N(\gamma)^{-1/2}\end{pmatrix}$ with $N(\gamma)>1$ and $\gamma_2$
is conjugate to $\begin{pmatrix} \epsilon(\gamma) & 0\\ 0 &
\epsilon(\gamma)\end{pmatrix}$ with $|\epsilon(\gamma)|=1$. For
$m\geq 1$, Selberg  \cite{Selberg} defines a Zeta function (see also
\cite{MoscoviciStanton91})
\begin{equation}\label{e:Zeta}
Z_m(s,\Gamma)=\prod_{\{\gamma\}^*_\Gamma}\mathop{\prod_{\nu=0}^\infty}_{|i|<
m }(1-\epsilon_\gamma^iN(\gamma)^{-s-\nu})^{-1}
\end{equation}
where the product is over all primitive conjugacy classes of mixed
elements in $\Gamma$. He shows that  $Z_m(s,\Gamma)$ is entire
(except when $m=1$ where it has a simple pole at $s=1$) and
satisfies a functional equation  relating $s$ and $1-s$. Its zeros
are either real in $\{-k\}_{k>0}\cup(-1,1)$ or complex in
$\frac{1}{2}+i\bbR$. They correspond to the eigenvalues of the
Casimir operator acting on suitable functions on $\Gamma\bs G$. As
Selberg remarks, the form that these Zeta functions take is
qualitatively similar to the Riemann Zeta function. In fact more so
then the case of one upper half plane where the corresponding
definition to (\ref{e:Zeta}) doesn't have a $-1$ in the exponent
(this feature is connected with the parity of $d$). If $\Gamma$ is a
congruence group and the Ramanujan-Selberg conjecture is true then
$Z_m(s,\Gamma)$ satisfies the ``Riemann hypothesis'', that is all
its non trivial zeros are on $\Re(s)=\frac{1}{2}$.

As a corollary of Theorem \ref{t:gap} we get a zero free region that
holds for all (but finitely many) of these Zeta functions.
\begin{cor}\label{c:zeta}
Given $t_0>\frac{5}{6}$ there is $m_0(\Gamma)$ such that
$Z_m(s,\Gamma)$ has no zeros in $\Re(s)>t_0$, for $m\geq
m_0(\Gamma)$.
\end{cor}

We now outline the main ideas of the proof of Theorem \ref{t:gap} (for the case $d=2$).
As mentioned above it is sufficient
to give a proof for $\Gamma$ a lattice derived from a quaternion
algebra, $\calA$, defined over a number field $L$ and an arbitrary
unitary representation $\rho$ of $\Gamma$. What we will actually
show, is that if a representation $\pi\cong\pi_1\otimes\pi_2$ occurs
in the decomposition of $L^2(\Gamma\bs G,\rho)$ with $p(\pi)$
sufficiently large, then all the spectral parameters of $\pi$  are
bounded.

We assume that $\pi\cong\pi_1\otimes\pi_2$ occurs in the
decomposition with say $p(\pi_1)>6$ and the other spectral parameter
large and get a contradiction: From our assumption $\pi_1\cong
\pi_{s_1}$ is complementary with
$|s_1-\frac{1}{2}|\in(\frac{1}{3},\frac{1}{2})$ and the second
factor is either principal series $\pi_2\cong\pi_{s_2}$ with
$s_2=\frac{1}{2}+ir_2,\; r_2\in[T,2T]$, or discrete series
$\pi_2\cong\frakD_{m}$ with weight $m\in[T,2T]$ for some large $T$.
Let $g_1,g_2 \in C^\infty(\bbR)$ be smooth even real valued
compactly supported functions such that their Fourier transforms
$h_j=\hat{g}_j$ are positive on $\bbR\cup i\bbR$. Further assume
that $h_2$ vanishes at zero to a large order (for the discrete
series case instead of $h_2$ we will use $\psi\in C^\infty(\bbR)$
that is smooth, positive and compactly supported away from zero).
For $T$ large and $R=c\log(T)$ we have
$h_1(Rr_1)h_2(\frac{r_2}{T})\gg \frac{T^{c|1/2-s_1|}}{\log(T)}$
(equivalently in the second case the same bound holds for
$h_1(Rr_1)\psi(\frac{m}{T})$).  From the positivity assumption, this
lower bound also holds when summing over all representation in the
decomposition (in the second case we also sum over all weights
$m\in\bbZ$). For the full sum we can also give an upper bound of
order $O_\epsilon(T^2+T^{c/2+\epsilon-1})$. For $c=6-2\epsilon$ and
$T$ sufficiently large the upper bound is already smaller then the
lower bound excluding the existence of such a representation in the
decomposition.

\begin{rem}\label{r:vanish}
When summing over all representation the trivial representation
$r_0=(i/2,i/2)$ also contributes. If $h_2$ (equivalently $\psi$) did
not vanish at zero then the trivial representation would contribute
$\sim T^{c/2}$ which is already larger then the lower bound coming
from the representation we wish to exclude. Hence, in order for this
strategy to have any chance of working we must make the function
$h_2$ vanish at zero (or respectively take $\psi$ supported away
from zero).
\end{rem}

To obtain the upper bound for the full sum we use the Selberg Trace
formula to transform the spectral sum to a sum over the conjugacy
classes (when summing over the weights we also use Poisson
summation). We then bound each summand by its absolute value. (Even
though the summands here are not positive, it turns out that the
oscillations are sufficiently slow so that we apparently don't lose
to much.) After some standard manipulation, using the fact that the
test functions are compactly supported, estimating the sum over the
conjugacy classes amounts to two counting arguments. The first is an
estimate on the number of algebraic integers in $L$ (viewed as a
lattice in $\bbR^n$) that lie inside a long and narrow rectangular
box whose sides are parallel to the coordinates axes. Using a simple
Dirichlet box principle argument we bound the number of such lattice
points by the volume of the box. The second counting problem is
counting the number of conjugacy classes in $\Gamma$ with a given
trace, which amounts to estimating the number of optimal embeddings
of certain orders into the quaternion algebra. This in turn is
translated (via the work of Eichler) to estimates of class numbers
of quadratic extensions of the number field $L$, that we obtain
using Dirichlet's class number formula.

\subsection*{Acknowledgements} We thank A.Gamburd and T.Venkataramana
for discussions about various aspects of the paper.

\section{Background and Notation}\label{s:Note}
In this section we go over some necessary background on lattices
$\Gamma$ in $G=\PSL(2,\bbR)^d$, on the spectral decomposition of
$L^2(\Gamma\bs G)$ and the Selberg trace formula.

\subsection{Irreducible lattices}
A discrete subgroup $\Gamma\subset G=\PSL(2,\bbR)^d$ is called a
lattice if the quotient $\Gamma\bs G$ has finite volume, and
co-compact when $\Gamma\bs G$ is compact. We say that a lattice
$\Gamma\subset G$ is irreducible, if for every (non-central) normal
subgroup $N\subset G$ the projection of $\Gamma$ to $G/N$ is dense.
An equivalent condition for irreducibility is that for any
nontrivial $1\neq\gamma\in\Gamma$, none of the projections
$\gamma_j\in G_j$ are trivial \cite[Theorem 2]{Shimizu}. Examples of
irreducible lattices can be constructed from norm one elements of
orders in a quaternion algebra (see below).

Recall that a nontrivial element $g\in\PSL(2,\bbR)$ is called
hyperbolic if $|\Tr(g)|>2$, elliptic if $|\Tr(g)|<2$, and parabolic
if $|\Tr(g)|=2$. For any nontrivial $1\neq\gamma\in\Gamma$, the
projections to the different factors are either hyperbolic or
elliptic. The irreducibility implies that there are no trivial
projections and since we assume $\Gamma$ is co-compact there are no
parabolic projections. There are purely hyperbolic elements (where
all projections are hyperbolic), and mixed elements (where some
projections are hyperbolic and other are elliptic). There could also
be a finite number of torsion points that are purely elliptic (see
e.g., \cite{Efrat87,Shimizu}).

\subsection{Lattices derived from quaternion algebras}\label{s:QuatAlg}
Let $L$ be a totally real number field and denote by
$\iota_1,\ldots,\iota_n$ the different embeddings of $L$ into
$\bbR$. Let $\calA$ be a quaternion algebra over $L$, ramified in
all but $d$ of the real places (say $\iota_1,\ldots,\iota_d$). That
is we have that $\calA\otimes_{\iota_j(L)}\bbR\cong\Mat_2(\bbR)$ for
$j\leq d$ and it is isomorphic to the standard Hamilton's
quaternions for $j>d$. Let $\calR$ be a maximal order inside
$\calA$, and denote by $\calR^1$ the group of (relative) norm one
elements inside this order. The image $\iota_j(\calR^1)\subset
\SL(2,\bbR)$ for $j\leq d$ and $\iota_j(\calR^1)\subseteq SU(2)$ for
$j>d$. The group
$\Gamma(\calR)=\set{(\iota_1(\alpha),\ldots,\iota_d(\alpha))\in\PSL(2,
\bbR)^d|\alpha\in
\calR^1}$ is a lattice inside $\PSL(2,\bbR)^d$ and it is co-compact
unless $n=d$ and $\calA=\Mat(2,L)$ (see \cite{Shavel76,Shimizu}).
Margulis's arithmeticity theorem \cite[Chapter IX]{Margulis91}
together with Weil's classification of arithmetic lattices
\cite{Weil60} implies that, up to commensurability, these are the
only examples of irreducible co-compact lattices in
$\PSL(2,\bbR)^d,\;d\geq 2$.

\subsection{Spectral decomposition}
Let $\Gamma$ be an irreducible co-compact lattice in $G$ and let
$\rho$ be a finite dimensional unitary representation of $\Gamma$.
The space $L^2(\Gamma\bs G,\rho)$ is the space of Lebesgue
measurable vector valued functions on $G$ satisfying that $f(\gamma
g)=\rho(\gamma)f(g)$ and that $\int_{\Gamma\bs G}|f(g)|^2dg\leq
\infty$. The group $G$ acts on $L^2(\Gamma\bs G,\rho)$ by right
multiplication and we can decompose it into irreducible
representations
\[L^2(\G\bs G,\rho)\cong \bigoplus \pi_k.\]
Any irreducible unitary representation $\pi_k$ is a product
$\pi_k\cong \pi_{k,1}\otimes\pi_{k,2}\ldots\otimes \pi_{k,d}$ where
the $\pi_{k,j}$'s are irreducible unitary representations of
$\PSL(2,\bbR)$. We briefly recall the classification of these
representations. Other then the trivial representation the
irreducible representations of $\PSL(2,\bbR)$ are either principal
series $\pi_{s},\;s\in\frac{1}{2}+i\bbR$, complementary series
$\pi_s,\;s\in (0,1)$, or discrete series $\frakD_{m},\;m\in\bbZ$.
The discrete and principal series are both tempered, while the
complementary series is non-tempered with
$p(\pi_s)=\max\set{\frac{1}{s},\frac{1}{1-s}}$. For a representation
$\pi_k\cong \pi_{k,1}\otimes\pi_{k,2}\ldots\otimes \pi_{k,d}$ of $G$
we have that $p(\pi_k)=\max_j p(\pi_{k,j})$.

\subsection{The Selberg Trace Formula}\label{s:trace}
The Selberg trace formula relates the spectral decomposition of
$L^2(\Gamma\bs G,\rho)$, to the conjugacy classes in $\Gamma$. We
refer to \cite[Sections 1-6]{Efrat87}, \cite[Chapter 3]{Hejhal76}
and \cite{Selberg} for the full derivation of the trace formula in
this setting.

Fix a wight $m\in \bbZ^d$. For simplicity, we assume that $m_j=0$
for $j\leq d_0$ and that $|m_j|>1$ for $j>d_0$. Denote by
$L^2(\Gamma\bs G,\rho,m)$ the subspace of $L^2(\Gamma\bs G,\rho)$
such that $\pi\cong \pi_1\otimes\pi_2\cdots\otimes\pi_d$ occurs in
the decomposition if and only if $\pi_j$ is principal or
complementary series for $j\leq d_0$ and $\pi_j\cong \frakD_{m_j}$
for $j>d_0$. Consider the decomposition
 \[L^2(\Gamma\bs G,\rho,m)\cong\bigoplus_{k=0}^\infty \pi_k,\]
into irreducible representations. For any $j\leq d_0$ let
$s_{k,j}=\frac{1}{2}+ir_{k,j}$ such that $\pi_{k,j}=\pi_{s_{k,j}}$.
For any $j\leq d_0$ let $g_j\in C^\infty(\bbR)$ be a smooth even
real valued compactly supported function, and let $h_j=\hat g_j$ be
its Fourier transform.  Recall that for any $\gamma\in\Gamma$ its
projections to the different factors are either hyperbolic,
$\gamma_j\sim \begin{pmatrix} N(\gamma_j)^{1/2} & 0\\ 0 &
N(\gamma_j)^{-1/2}\end{pmatrix}$ with $N(\gamma_j)=e^{l_j}>1$, or elliptic
$\gamma_j\sim \begin{pmatrix} \epsilon(\gamma_j) & 0\\
0 & \epsilon(\gamma_j) \end{pmatrix}$ with
$\epsilon(\gamma_j)=e^{i\theta_j}\in S^1$. Define the functions
$\tilde{h}_j(\gamma_j)$ by
\[\tilde{h}_j(\gamma_j)=\frac{g(l_{j})}{\sinh(l_{j}/2)},\]
when $\gamma_j$ is hyperbolic, and
\[\tilde{h}_j(\gamma_j)=\frac{1}{\sin\theta_j}\int_{-\infty}^{\infty}
\frac{\cosh[(\pi-2\theta_j)r]}{\cosh(\pi r)}h(r)dr\] when $\gamma_j$
is elliptic. The Selberg trace formula, applied to the product
$h(r)=\prod_{j\leq d_0} h_j(r_j)$, then takes the form
\begin{eqnarray*}
\lefteqn{\sum_k h(r_k)=}\\
&&\frac{\vol(\Gamma\bs G)\chi_\rho(1)}{(4\pi)^d}\prod_{j\leq
d_0}\left(\int_{\bbR} h_j(r_j)r_j\tanh(\pi
r_j)dr_j\right)\prod_{j>d_0}(2|m_j|-1)\\
&&+\sum_{\{\gamma\}}\vol(\Gamma_\gamma\bs
G_\gamma)\chi_\rho(\gamma)\prod_{j\leq d_0}
\tilde{h}_j(\gamma_j)\prod_{j>d_0}\frac{e^{\pm
2i|m_j|\theta_j}}{1-e^{\pm 2i\theta_j}},
\end{eqnarray*}
where the sum on the right hand side is over all $\Gamma$-conjugacy
classes $\{\gamma\}\in\Gamma^\sharp$ that are elliptic for $j>d_0$,
where $G_\gamma$ denotes the centralizer of $\gamma$ in $G$ and
$\Gamma_\gamma=G_\gamma\cap \Gamma$,
$\chi_\rho(\gamma)=\Tr(\rho(\gamma))$ is the character of the
representation, and the $\pm$ signs are determined by the signs of
the $m_j$'s.

\section{Proof of Theorem \ref{t:generic}}
In this section we give the proof of Theorem \ref{t:generic}. We
consider $\SU(2)\times\SU(2)$ as a deformation space for lattices
$\G$ in $G=\SL(2,\bbR)\times \SU(2)$. We construct a dense set of
irreducible lattices inside this deformation space, each satisfying
that $E(\G\bs G)$ is infinite, and then use these to show that the
same is true generically.

\subsection{Deformation space}
Let $\G$ be an irreducible lattice in $G=\SL(2,\bbR)\times \SU(2)$.
The projection $P_1$ of $\G$ onto the first factor has image
$\Lambda$ which is a lattice in $\SL(2,\bbR)$. For the purpose of
constructing lattices $\G$ in $G$ with $|E(\G\bs G)|=\infty$, we
assume that $P_1:\G\to \Lambda$ is an isomorphism. In this way we
can identify
$$\G=\set{(\gamma,\rho(\gamma)):\gamma\in\Lambda},$$
where $\rho=P_1\circ P_2^{-1}:\Lambda\hookrightarrow\SU(2)$. For
$\Lambda$ we take the congruence subgroup $\G(2)$ of $\SL(2,\bbR)$
which is a free group on two generators $A=\begin{pmatrix} 1 & 2\\ 0
& 1\end{pmatrix}$, $B=\begin{pmatrix} 1 & 0\\ 2 & 1\end{pmatrix}$.
Our deformation space of such lattices can then be described as
$\SU(2)\times\SU(2)$ where for any $u=(u_1,u_2)\in
\SU(2)\times\SU(2)$ define $\rho_u$ by
$\rho_u(A)=u_1,\;\rho_u(B)=u_2$ extended to a homomorphism of
$\Lambda$ into $\SU(2)$ and let $\G_u=(\Lambda,\rho_u)$. One can
further identify such lattices in $G$ which are conjugate in $G$ but
for our analysis there is no need to do so. For any $n\geq 3$ we
choose $u_1,u_2$ so that $u_1^n=u_2^n=1$ and satisfy no further
relations (that is the corresponding image $\rho_u(\Lambda)$ is
isomorphic to the free product $(\bbZ/n\bbZ)*(\bbZ/n\bbZ)$). Varying
over all such $\rho_u$ and all $n>3$ yields a dense subset in our
deformation space. Note that for any such choice of $u$ the image
$\rho_u(\Lambda)$ is dense in $\SU(2)$, that is,
$\G_u=(\Lambda,\rho_u)$ is irreducible. We will now show that for
such a lattice we have $|E(\G_u\bs G)|=\infty$.
\begin{thm}\label{t:Einf2}
For any homomorphism $\rho_u:\Lambda\to\SU(2)$ as above the
corresponding lattice $\G_u=(\Lambda,\rho_u)$  satisfies $|E(\G_u\bs
G)|=\infty$.
\end{thm}

\subsection{Spectral theory for infinite volume quotients of $\bbH$}
For the proof of Theorem \ref{t:Einf2} we will make a reduction to
the spectral theory of $L^2(L\bs\bbH)$ with $\bbH$ the upper half
plane and $L=\ker\rho_u$ acting by linear fractional
transformations. Before proceeding with the proof we review some
facts on the spectral theory of these infinite volume hyperbolic
surfaces that we will need (we refer to \cite{Sullivan87} for
details). Let $L$ be a torsion free discrete subgroup of
$\SL(2,\bbR)$. Then $L\bs \bbH$ is a complete hyperbolic surface and
the Laplacian on smooth functions of compact support on $L\bs \bbH$
has a unique self adjoint extension denoted by $\lap$. Let
$\lambda_0(L\bs\bbH)$ denote the bottom of the spectrum of $\lap$ so
that the spectrum is contained in $[\lambda_0,\infty)$. Closely
related to $\lambda_0(L\bs \bbH)$ is the exponent of convergence
$\delta(L)\in [0,1]$ (see \cite[page 333]{Sullivan87} for
definition). When $L$ is nonelementary and contains a parabolic
element this exponent $\delta(L)>\frac{1}{2}$ \cite[Theorem
7]{Beardon68}, in which case the Elstrodt-Patterson Theorem
\cite[Theorem 2.17]{Sullivan87} says that
$\lambda_0(L\bs\bbH)=\delta(L)(1-\delta(L))$, and in particular
$\lambda_0(L\bs\bbH)<1/4$.

We shall be interested in the case where $L$ is a normal subgroup of
$\Lambda$ and $\Lambda/L$ is not amenable. In this case Brooks
\cite{Brooks85} shows  that $\lambda_0(L\bs\bbH)>0$. Summarizing the
above remarks we have
\begin{prop}
Let $L$ satisfy that $\Lambda/L$ is not amenable and $\delta(L)>1/2$
then $0<\lambda_0(L\bs\bbH)<1/4$ and $\lambda_0$ is an accumulation
point of distinct points of the spectrum of $\lap$ on $L^2(L\bs
\bbH)$.
\end{prop}
\begin{proof}
From the above remarks it is clear that $0<\lambda_0(L\bs\bbH)<1/4$.
We will show that there is no eigenfunction in $L^2(L\bs\bbH)$ with
eigenvalue $\lambda_0$ implying that $\lambda_0$ cannot be an
isolated point in the spectrum. We recall that if an eigenfunction
$\phi\in L^2(L\bs \bbH)$ with eigenvalue $\lambda_0$ exists then it
is unique up to a nonzero scalar multiple \cite[Corollary
2.9]{Sullivan87}. On the other hand, as $L$ is normal in $\Lambda$,
for any $\gamma\in \Lambda$ the function $\phi(\gamma z)\in L^2(L\bs
\bbH) $ is also a $\lambda_0$-eigenfunction. Consequently, we must
have $\phi( z)=\phi(\gamma z)$ for all $\gamma\in \Lambda$ and since
$\Lambda/L$ is infinite then $\phi$ can not be in $L^2(L\bs\bbH)$.
\end{proof}
\begin{rem}
The situation here is very different from the case of geometrically
finite quotients where Lax and Phillips \cite{LaxPhillips82} showed
that the point spectrum is finite. Indeed, we recall that a finitely
generated normal subgroup of a free group is always of finite index
\cite{KarrassSolitar57}. Hence, the assumption that $L$ is a normal
subgroup with infinite index in $\Lambda$ implies that $L$ must be
infinitely generated and in particular not geometrically finite.
\end{rem}

\subsection{Construction of nontempered points}
Fix $n>3$ and a homomorphism  $\rho:\Lambda\to \SU(2)$ with
$\rho(\Lambda)\cong (\bbZ/n\bbZ)*(\bbZ/n\bbZ)$ such that
$\rho(A)^n=\rho(B)^n=1$. The kernel $L=\ker(\rho)$ is normal in
$\Lambda$ and $\Lambda/L\cong (\bbZ/n\bbZ)*(\bbZ/n\bbZ)$ is
infinite (and not amenable). Also $A^n\in L$ is parabolic so
$\delta(L)>\frac{1}{2}$
and hence $0<\lambda_0(L\bs \bbH)<\frac{1}{4}$ is an accumulation
point of distinct points in the spectrum.

Now, for $l\geq 0$ let $\sigma_l=\mathrm{sym}^l$ denote the $l+1$
dimensional irreducible representation of $\SU(2)$. According to
Weyl $L^2(\SU(2))$ decomposes under the regular representation as
\[L^2(\SU(2))=\bigoplus_{l=0}^\infty (\dim\sigma_l)W_l,\]
where $W_l\cong\sigma_l$. Correspondingly the regular representation
of $G$ on $L^2(\G\bs G)$ decomposes into the representations
$L^2(\Lambda\bs \SL(2,\bbR),\sigma_l\circ\rho)$ each occurring with
multiplicity $l+1$. Here
\begin{eqnarray}\label{e:sigma}
\lefteqn{L^2(\Lambda\bs \SL(2,\bbR),\sigma_l\circ\rho)=}\\
\nonumber &&\set{F:\SL(2,\bbR)\to
\bbC^{l+1}|F(\gamma g)=\sigma_l(\rho(\gamma))F(g),\;\gamma \in \Lambda},
\end{eqnarray}
with the right action of $\SL(2,\bbR)$ (we may normalize so that
$\sigma_l$ acts unitarily on $\bbC^{l+1}$ with respect to the
standard inner product). Since we are only interested in
representations $\pi$ of $\SL(2,\bbR)$ appearing in (\ref{e:sigma})
which are nontempered, we may restrict to $\pi$'s which are
spherical.

Denote by $L^2(\Lambda\bs\bbH,\sigma_l\circ\rho)$ the space of
square integrable vector valued functions on the upper half plane
$\bbH$, satisfying $F(\gamma z)=\sigma_l(\rho(\gamma))F(z)$, where
$\gamma$ acts on $z\in\bbH$ by fractional linear transformations.
This space is naturally identified with the space of spherical
vectors in $L^2(\Lambda\bs\SL(2,\bbR),\sigma_l\circ\rho)$. Let
$F_{1,l},F_{2,l},\ldots$ in $L^2(\Lambda\bs\bbH,\sigma_l\circ\rho)$
be an orthonormal basis of eigenvectors of $\lap$ with eigenvalues
$\lambda_{j,l}=\frac{1}{4}+t^2_{j,l}$ giving the discrete spectrum
and $E(z,\frac{1}{2}+it),\;t\in\bbR$ spanning the (tempered)
continuous spectrum. Note that if $\lap
F_{j,l}+\lambda_{j,l}F_{j,l}=0$ with
$\lambda_{j,l}=s_{j,l}(1-s_{j,l})<\frac{1}{4}$ then there is a
nontempered representation $\pi$ appearing in $L^2(\Gamma\bs G)$
with $p(\pi)=\frac{1}{1-s_{j,l}}$. Hence showing that $|E(\G\bs
G)|=\infty$ is equivalent to showing that there are infinitely many
distinct eigenvalues $\lambda_{j,l}$ below $1/4$. The following
proposition then concludes the proof of Theorem \ref{t:Einf2}.
\begin{prop}
With the above notations, there are infinitely many eigenvalues
$\lambda_0(L\bs\bbH)<\lambda_{j,l}<\frac{1}{4}$ accumulating at
$\lambda_0(L\bs\bbH)$.
\end{prop}
\begin{proof}
Let $k(z,w)$ be a point pair invariant on $\bbH$ as in
\cite{Selberg65} (i.e., for any $g\in
\SL(2,\bbR),\;k(gz,gw)=k(z,w)$). We assume that for $z$ fixed
$k(z,w)$ is a continuous compactly supported function in $w$. We
have the spectral expansion for the kernel $K_{\sigma_l}(z,w)$ (see
\cite[Chapter 8, equation 4.1]{Hejhal83}) given by
\begin{eqnarray}\label{e:spec1}
K_{\sigma_l}(z,w)&=& \sum_{\gamma\in\Lambda}k(\gamma
z,w)\sigma_l\circ\rho(\gamma)\\
\nonumber &=& \sum_{j=1}^\infty h(t_{j,l})F_{j,l}(z)\overline{F_{j,l}^t(w)}\\
\nonumber &+& \int_\bbR
h(t)E(z,\frac{1}{2}+it)\overline{E^t(w,\frac{1}{2}+it)}dt,
\end{eqnarray}
where $h(s)=\int_\bbH k(i,z)y^s\frac{dxdy}{y^2}$ is the Selberg
transform of $k$. Note that for any fixed $z,w$ both sides are
$(l+1)\times(l+1)$ matrices. Taking traces of these matrices gives
\begin{eqnarray}\label{e:spec2}
 \sum_{\gamma\in\Lambda}k(\gamma
z,w)\chi_l(\rho(\gamma))&=& \sum_{j=1}^\infty
h(t_{j,l})\pr{F_{j,l}(z)}{F_{j,l}(w)}\\
\nonumber &+& \int_\bbR
h(t)\pr{E(z,\frac{1}{2}+it)}{E(w,\frac{1}{2}+it)}dt,
\end{eqnarray}
where $\chi_l$ is the character of $\sigma_l$ on $\SU(2)$ and we
denote by $\pr{}{}$ the standard inner product on $\bbC^{l+1}$.

Let $\psi(z)$ be a continuous function of compact support in $\bbH$
and integrate (\ref{e:spec2}) against $\psi(z)\overline{\psi(w)}$ to
get
\begin{eqnarray}\label{e:spec3}
 \frac{1}{l+1}\sum_{\gamma\in\Lambda}\int_\bbH\int_\bbH
\psi(z)\overline{\psi(w)}
 k(\gamma z,w)dv(z)dv(w)\chi_l(\rho_n(\gamma))\\
 \nonumber=\int_C h(t)d\mu_l(t)
\end{eqnarray}
where $\mu_l$ is the positive measure on $C=[0,\infty)\cup
[0,\frac{i}{2}]$ given by

\begin{eqnarray}\label{e:mul}
\frac{1}{l+1}\sum_{j=1}^\infty
\pr{\int_\bbH \psi(z)F_{j,l}(z)dv(z)}{\int_\bbH
\psi(z)F_{j,l}(w)dv(w)}\delta_{t_{j,l}}\\
\nonumber + \frac{1}{l+1}\pr{\int_\bbH
\psi(z)E(z,\frac{1}{2}+it)dv(z)}{\int_\bbH
\psi(w)E(w,\frac{1}{2}+it)dv(w)}dt
\end{eqnarray}
Note that for fixed $k$ the sum over $\Lambda$ on the left hand side
of \ref{e:spec3} is finite. Also as $l\to\infty$ we have
$\frac{1}{l+1}\chi_l(u)\to 1$ if
$u=1$ and tends to $0$ if $u\neq 1$. Hence, taking the limit
$l\to\infty$ in (\ref{e:spec3}) (for $k$ and $\psi$ fixed) we get that
\begin{equation}\label{e:speclim}
\mu_l(h)\to \sum_{\gamma\in L}\int_\bbH\int_\bbH
\psi(z)\overline{\psi(w)} k(\gamma z,w)dv(z)dv(w).
\end{equation}
If the function $\psi(z)$ is supported in a small ball $B$ in $\bbH$
that is contained in one fundamental domain $\calF$ for $L\bs\bbH$
then we can think of $\psi$ also as an element of
$L^2(L\bs\bbH)$. For such $\psi$ we get
\begin{equation}\label{e:speclim1}
\mu_l(h)\to \int_\calF\int_\calF \psi(z)\overline{\psi(w)}
K_L(z,w)dv(z)dv(w),
\end{equation}
where \begin{equation}\label{e:kernel}
K_L(z,w)=\sum_{\gamma\in
L}k(\gamma z,w).
\end{equation}
The function $K_L(z,w)$ is $L\times L$ invariant and gives a kernel
for a bounded self-adjoint operator on $L^2(L\bs\bbH)$. The family
of such operators (when taking different point pair invariants $k$)
is a commutative algebra that also commutes with $\lap$.
Consequently, this whole algebra can be simultaneously diagonalized
together with $\lap$. For any fixed $\psi\in L^2(L\bs \bbH)$ there
is a corresponding positive spectral measure $\nu_\psi$ on the
spectrum of $\lap$. That is, using the parameter
$t=\sqrt{\lambda-1/4}$ we have the spectral decomposition
\begin{equation}\label{e:spec}
\pr{K_L\psi}{\psi}=\int_C h(t)d\nu_\psi(t).
\end{equation}
Consequently, from (\ref{e:speclim1}) and (\ref{e:spec}) we get that
for every function $h$ which is the Selberg transform of $k$
continuous of compact support (in particular for any even function
$h$ with Fourier transform smooth of compact support) as $l\to
\infty$
\begin{equation}\label{e:lim}
\mu_l(h)\to\nu_\psi(h).
\end{equation}

Now, since the spectrum of $\lap$ on $L^2(L\bs\bbH)$ has $\lambda_0$
as an accumulation point it follows that given $\epsilon>0$ we can
find a closed nonempty subinterval $I$ of
$(\lambda_0,\lambda_0+\epsilon)$ such that the spectral projector
$P_I$ onto $I$ is nonzero. Let $f$ be a nonzero element in the image
of this projector $P_I$. One can choose a small ball $B$ in $\bbH$
which is injective in $L\bs\bbH$ and such that $f$ restricted to $B$
is a nonzero $L^2$ function. Take $\psi$ to be supported in $B$,
continuous and such that its integral over $B$ against $f$ is not
zero. Then as members in $L^2(L\bs \bbH)$ the inner product of $f$
and $\psi$ is not zero so that the support of $\nu_\psi$ meets $I$
nontrivially. Let $J\subset(\lambda_0,\lambda_0+\epsilon)$ be an
interval strictly containing $I$ and let $h$ be an even function
with Fourier transform compactly supported that is negative outside
$J$ and satisfies that $\nu_\psi(h)>0$. Then from (\ref{e:lim})
(with this $\psi$ and $h$) we get that for sufficiently large $l$
the support of $\mu_l$ in (\ref{e:mul}) meets $J$ nontrivially.
Consequently, for all sufficiently large $l$ there is an eigenvalue
$\lambda_{j,l}\in (\lambda_0,\lambda_0+\epsilon)$. Repeating this
procedure (making $\epsilon$ smaller) will produce infinitely many
eigenvalues accumulating at $\lambda_0$.

To conclude the proof we give a construction for an even function
$h$ with Fourier transform smooth and compactly supported that is
negative outside $J$ and satisfies that $\nu_\psi(h)>0$. Fix a
smooth compactly supported function $g$ with Fourier transform
$\hat{g}$ even and positive on $C$ and set
$M>\frac{\int_{C}\hat{g}(t)d\nu_\psi(t)}{\int_{I}\hat{g}(t)d\nu_\psi(t)}$
(this is finite since the support of $\nu_\psi$ meets $I$). Now let
$F(t)=\sum_{n\leq N}a_n\cos(nt)$, be a trigonometric polynomial
satisfying that $F(t)>M$ for $t^2+\frac{1}{4}\in I$ and $-1<F(t)<0$
for $t^2+\frac{1}{4}$ in the complement of $J$. (The existence of
such a trigonometric polynomial is guaranteed by the Weierstrass's
approximation theorem for polynomials recalling that the Chebyshev
polynomials satisfy $T_n(\cos(t))=\cos(nt)$). Now the function
$h(t)=\hat{g}(t)F(t)$ has Fourier transform smooth of compact
support and satisfies $h(t)<0$ on the complement of $J$ (as it has
the same sign as $F$) and $\nu_\psi(h)>0$ (by the choice of $M$).
\end{proof}

We now complete the proof of Theorem \ref{t:generic}, showing that
for generic $u$ the exceptional exponent set $E(\G_u\bs G)$ is
infinite. As we noted and is easily shown, the set of $u$'s that we
consider in Theorem \ref{t:Einf2} are dense in $\SU(2)\times\SU(2)$.
Let $u_j,\;j=1,2,\ldots$ be an enumeration of a dense set of such
$u$'s. Now for each $l$, the spectrum in $[0,1/4]$ of $\lap$ on
$L^2(\Lambda\bs\bbH,\sigma_l\circ \rho)$ is continuous in $u$. Hence
it follows from Theorem \ref{t:Einf2} that for each $j=1,2,3,\ldots$
there is $\epsilon_j$ such that for $u$ in a small neighborhood
$B(u_j,\epsilon_j)$ of $u_j$ the lattice $\Gamma_u=(\Lambda,\rho_u)$
satisfies $|E(\G_u\bs G)|>j$. Now let
$$B=\bigcap_{J=1}^\infty\bigcup_{j=J}^\infty B(u_j,\epsilon_j).$$
Then $B$ is of the second category in $\SU(2)\times\SU(2)$ and for
any $u\in B$, $E(\Gamma_u\bs G)$ is infinite. We have thus shown
that a generic lattice in the sense of Baire has infinitely many
exceptional exponents. Note that for the generic $u\in
\SU(2)\times\SU(2)$,  $u_1$ and $u_2$ generate a free  group in
$\SU(2)$. Hence the limit measure in (\ref{e:lim}) (as $l\to\infty$)
for such a lattice is supported on $\bbR$ (i.e., it has no
exceptional spectrum). That is the generic lattice has infinitely
many exceptional exponents but in terms of density almost all the
representations are tempered.

\section{Proof of Theorem \ref{t:gap}}
We now give the proof of Theorem \ref{t:gap}. In order to simplify
notations we will write down the full details only for the case $d=2$.
The modifications required to handle $d>2$ are straight
forward and are accounted for in section \ref{s:generald}.

\subsection{Reduction to an asymptotic argument}
Fix a co-compact irreducible lattice, $\Gamma\subset \PSL(2,\bbR)^2$, derived from a
quaternion algebra and let $\rho$ be a unitary representation of
$\Gamma$. It is well known that there are only finitely many
representation occurring in $L^2(\Gamma\bs G,\rho)$ with all
spectral parameters bounded. We can thus reduce Theorem \ref{t:gap}
to the following asymptotic argument
\begin{thm}\label{t:gap1}
Assume that $\pi\cong\pi_1\otimes\pi_2$ occurs in $L^2(\G\bs
G,\rho)$ and that $\pi_1\cong \pi_{s_1}$ is complementary series and
$\pi_2$ is either principal $\pi_{s_2}$ with
$s_2=\frac{1}{2}+ir_2,\;r_2\in[T,2T]$ or discrete $\frakD_m$ with
$|m|\in[T,2T]$. Then for any $c>0$ (as $T\to\infty$)
\begin{eqnarray*}%\label{e:geometric}
T^{c|\frac{1}{2}-s_1|}\ll_\epsilon\dim(\rho)(\frac{T^2}{\log(T)}+T^{\frac{c}{2}
-1+\epsilon})
\end{eqnarray*}
\end{thm}

We now show that this asymptotic argument implies Theorem
\ref{t:gap}.
\begin{proof}[Proof of Theorem \ref{t:gap}]
Fix $\alpha,\epsilon>0$ and let
$M=M(\alpha,\epsilon)=\frac{6+\alpha}{\alpha-\epsilon(4+\alpha)}$.
Then by theorem \ref{t:gap1} with $c=6-2\epsilon$, there is a
constant $C=C(\epsilon,\Gamma)$ such that if $\pi=\pi_1\otimes\pi_2$
occurs in the decomposition with $\pi_1$ complementary with
$p(\pi_1)\geq 6+\alpha$ (i.e., $|\frac{1}{2}-s_1|\geq
\frac{4+\alpha}{2(6+\alpha)}$) then $\pi_2$ is either complementary,
or principal with parameter $r_2\leq (C\dim\rho)^M$ or discrete with
parameter $|m|\leq (C\dim\rho)^M$. Theorem \ref{t:gap} now follows
as there are at most $O((\dim\rho)^{2M})$ such representations.
\end{proof}

\subsection{Reduction to a counting argument}
We now use the Selberg trace formula to reduce Theorem \ref{t:gap1}
to a counting argument.
\begin{prop}\label{p:geom}
Assume that $\pi\cong\pi_1\otimes\pi_2$ occurs in $L^2(\G\bs
G,\rho)$ and satisfies the hypothesis of Theorem \ref{t:gap1}. Then
for any $c>0$ as $T\to\infty$
\begin{eqnarray*}
T^{c|\frac{1}{2}-s_{1}|}\ll_\epsilon \dim(\rho)(
\frac{T^2}{\log(T)}+T\!\!\!\!\!\!\!\!\!\mathop{\sum_{|t_1|\leq
T^{c/2}}}_{|t_2|=2+O(T^{-2+\epsilon})}\!\!\frac{F_\Gamma(t)}{\sqrt{
(t_1^2-4)(t_2^2-4)}}\\
+\frac{1}{T}\!\!\mathop{\sum_{|t_1|\leq T^{c/2}}}_{|t_2|\leq
2}\!\!\frac{F_\Gamma(t)}{\sqrt{(t_1^2-4)(t_2^2-4)}}).
\end{eqnarray*}
where the summation is over elements $t=(t_1,t_2)\in\Tr(\Gamma)$ and
$$F_\Gamma(t)=\mathop{\sum_{\{\gamma\}}}_{\Tr(\gamma)=t}\vol(\Gamma_\gamma\bs
G_\gamma),$$
is counting the number of conjugacy classes in $\Gamma$ with a given
trace.
\end{prop}
We will give the proof separately for the two cases when $\pi_2$ is
principal or discrete series.
\begin{proof}[\textbf{Proof for principal series}]
Let $\pi\cong\pi_1\otimes\pi_2$ occur in $L^2(\G\bs G,\rho)$ with
$\pi_1\cong \pi_{s_1}$ complementary series and $\pi_2\cong
\pi_{s_2}$ principal with parameter $r_2\in[T,2T]$. Consider the
function
$$h_{R,T}(r_1,r_2)=h_1(Rr_1)h_2(\frac{r_2}{T}),$$ where $h_1,h_2$ are even
positive functions with Fourier transforms $g_1,g_2$ smooth and supported on
$[-1,1]$. We also assume that $h_2$ vanishes at zero to a large order
$>2/\epsilon$. We note that this vanishing assumption is crucial for the proof
(see remark \ref{r:vanish}).
Note that for $s_1=\frac{1}{2}+ir_1\in(0,\frac{1}{2})$ and
$r_{2}\in[T,2T]$ we can bound the function $h_{R,T}(r)\gg
\frac{\exp(R|\frac{1}{2}-s_1|)}{R}$ from below. Since the function
is positive, this is also a lower bound for the sum over the full
spectrum
\[\frac{\exp(R|\frac{1}{2}-s_1|)}{R}\ll\sum_k h_{R,T}(r_k).\]
We now use the trace formula (with wight $(0,0)$) to transform the
sum over the eigenvalues to a sum over conjugacy classes. The
geometric side of the trace formula is given by
\begin{eqnarray*}
\frac{\vol(\Gamma\bs G) \chi_\rho(1)}{16\pi^2}{\int\!\!\!\int}_{\!\!\!\bbR^2}
h_{R,T}(r_1,r_2)r_1\tanh(\pi r_1)r_2\tanh(\pi r_2)dr_1dr_2\\
+ \frac{T}{R}\!\!\sum_{\{\gamma\}\in e.h}\!\!\!\vol(\G_\gamma\bs
G_\gamma) \frac{\chi_\rho(\gamma)g_2(T
l_{\gamma_2})}{\sinh(\frac{l_{\gamma_2}}{2})\sin\theta_{\gamma_1}}\int_{-\infty}
^{\infty}
\frac{\cosh(\frac{(\pi-2\theta_{\gamma_1})r}{R})}{\cosh(\frac{\pi
r}{R})}h_1(r)dr \\
+\frac{T}{R}\!\!\sum_{\{\gamma\}\in h.e}\!\!\!\vol(\G_\gamma\bs
G_\gamma)\frac{\chi_\rho(\gamma)
g_1(\frac{l_{\gamma_1}}{R})}{\sinh(\frac{l_{\gamma_1}}{2})\sin\theta_{\gamma_2}}
\int_{-\infty}^{\infty}\frac{\cosh[(\pi-2\theta_{\gamma_2})Tr]}{\cosh(\pi T
r)}h_2(r)dr\\
+\frac{T}{R}\sum_{\{\gamma\}\in h.h}\vol(\G_\gamma\bs
G_\gamma)\frac{\chi_\rho(\gamma)g_1(\frac{l_{\gamma_1}}{R})}{\sinh(\frac{l_{
\gamma_2}}{2})} \frac{g_2(Tl_{\gamma_1})}{\sinh(\frac{l_{\gamma_2}}{2})}\\
\end{eqnarray*}
where we divided the conjugacy classes into the different types:
Trivial conjugacy class, elliptic-hyperbolic, hyperbolic-elliptic
and hyperbolic-hyperbolic. (There could also be elliptic-elliptic
elements that we ignore as their total contribution to the sum is
bounded by $O(1)$.) We will now give separate bounds for each term
where we replace each summand by its absolute value and bound the
character of the representation $|\chi_\rho(\gamma)|\leq
\chi_\rho(1)=\dim\rho$ by the dimension.

\subsubsection*{Trivial Conjugacy class}
By making a change of variables $r_1\mapsto \frac{r_1}{R}$ and
$r_2\mapsto Tr_2$ and bounding $|\tanh(t)|\leq 1$, the contribution
of the trivial conjugacy class is bounded by
$O(\chi_\rho(1)\frac{T^2}{R^2})$.
\subsubsection*{Elliptic-Hyperbolic}
For the elliptic-hyperbolic conjugacy class, note that $g_2$ is
supported on $[-1,1]$, hence the only conjugacy classes contributing
to this sum are the ones with $l_{\gamma_2}\leq \frac{1}{T}$. But
there are only finitely many conjugacy classes with $\gamma_1$
elliptic and $l_{\gamma_2}\leq\frac{1}{T}$, hence the contribution
of these conjugacy classes is bounded by $O(T)$. (In fact for $T$
sufficiently large there are no conjugacy classes satisfying this
condition so that it is bounded by $O(1)$).
\subsubsection*{Hyperbolic-Hyperbolic}
The only contribution of hyperbolic-hyperbolic elements comes from
elements with $l_{\gamma_1}\leq R$ and $l_{\gamma_2}\leq
\frac{1}{T}$. It is convenient to rephrase this in terms of the
traces of the conjugacy classes. For each conjugacy class,
$\{\gamma\}$ it's trace $t=(t_1,t_2)=(\Tr(\gamma_1),\Tr(\gamma_2))$
is given by $t_j=e^{l_{\gamma_j}/2}+e^{-l_{\gamma_j}/2}$.
Consequently, the only contribution comes from conjugacy classes
such that $|t_1|\sim e^{l_{\gamma_1}/2}\leq e^{R/2}$ and
$2<|t_2|\leq 2+\frac{1}{T^2}$. We can also write
$\sinh(l_{\gamma_j}/2)=\sqrt{t_j^2-4}$ so the contribution of the
hyperbolic-hyperbolic conjugacy classes is bounded by
\[(h.h.)\ll\dim(\rho)\frac{T}{R}\mathop{\sum_{|t_1|\leq
e^{R/2}}}_{2<|t_2|<2+\frac{1}{T^2}}\frac{F_\Gamma(t)}{\sqrt{(t_1^2-4)(t_2^2-4)}}
.\]
\subsubsection*{Hyperbolic-Elliptic}
As above, since $g_1$ is supported on $[-1,1]$ the only contribution
here is from conjugacy classes satisfying $l_{\gamma_1}\leq R$. For
these we estimate the contribution of the integral
\[\int_{-\infty}^{\infty}
    \frac{\cosh[(\pi-2\theta)Tr]}{\cosh(\pi T r)}h_2(r)dr.\]
First for $\theta=\theta_{\gamma_2} <T^{-1+\epsilon}$ we just bound
this integral by $O(1)$. Next, for $\theta>T^{-1+\epsilon}$ separate
this integral into two parts: The first when $r$ is small, where we
just bound $|\frac{\cosh[(\pi-2\theta)Tr]}{\cosh(\pi T r)}|\leq 1$
to get
\[|\int_{|r|\leq T^{-\epsilon/2}}\frac{\cosh[(\pi-2\theta)Tr]}{\cosh(\pi T
r)}h_2(r)dr|\leq \int_{|r|\leq T^{-\epsilon/2}}h_2(r)dr.\]
Since we assume $h_2$ vanishes at zero to order$>\frac{2}{\epsilon}$
we get that $h_2(r)\ll r^{\frac{2}{\epsilon}}$ near zero, hence, the
contribution of this part is bounded by $O(T^{-2})$. Now for the
next part we can use the exponential decay of
$\frac{\cosh[(\pi-2\theta)Tr]}{\cosh(\pi T r)}$ to get that
$\int_{|r|>T^{-\epsilon/2}}\frac{\cosh[(\pi-2\theta)Tr]}{\cosh(\pi T
r)}h_2(r)dr \ll e^{-T^{\epsilon/2}}$, so that for large $T$ the
whole integral is bounded by $O(\frac{1}{T^2})$.

Thus, for $|\theta_{\gamma_2}|<T^{-1+\epsilon}$ and
$l_{\gamma_1}\leq R$ (equivalently
$2-\frac{1}{T^{2-\epsilon}}\leq|t_2|\leq 2$ and $|t_1|\leq e^{R/2}$)
we get a contribution of
\[\frac{T}{R}\frac{c_\gamma|\chi_\rho(\gamma)|}{\sinh(l_{\gamma_1}/2)\sin\theta_
{\gamma_2}}=\frac{T}{R}\frac{c_\gamma
|\chi_\rho(\gamma)|}{\sqrt{(t_1^2-4)|t_2^2-4|}},\]
and for $|\theta_{\gamma_1}|>T^{-1+\epsilon}$ and $l_{\gamma_2}\leq
R$ (equivalently $2-\frac{1}{T^{2-\epsilon}}\geq|t_2| \leq 2$ and
$|t_1|\leq e^{R/2}$) we get a contribution of
\[\frac{1}{RT}\frac{c_\gamma
|\chi_\rho(\gamma)|}{\sinh(l_{\gamma_1}/2)\sin\theta_{\gamma_2}}=\frac{1}{RT}
\frac{c_\gamma|\chi_\rho(\gamma)|}{\sqrt{(t_1^2-4)|t_2^2-4|}}.\]
We can thus bound the contribution of the  hyperbolic-elliptic
elements by
\begin{eqnarray*}(h.e.) &\ll&\dim(\rho)\frac{T}{R}\mathop{\sum_{|t_1|\leq
e^{R/2}}}_{2-\frac{1}{T^{2-\epsilon}}<|t_2|<2}\frac{F_\Gamma(t)}{\sqrt{
(t_1^2-4)(t_2^2-4)}}.\\
&+&\dim(\rho)\frac{1}{RT}\mathop{\sum_{|t_1|\leq
e^{R/2}}}_{|t_2|\leq 2}\frac{F_\Gamma(t)}{\sqrt{(t_1^2-4)(t_2^2-4)}}
\end{eqnarray*}
Putting all these bounds together, and taking $R=c\log(T)$ concludes
the proof.
\end{proof}

\begin{proof}[\textbf{Proof for discrete series}]
Let $\pi\cong\pi_1\otimes\pi_2$ occur in $L^2(\G\bs G,\rho)$ with
$\pi_1\cong \pi_{s}$ complementary series with
$s=\frac{1}{2}+ir\in(0,\frac{1}{2})$ and $\pi_2\cong \frakD_m$
discrete series with weight $m\in[T,2T]$ (the case of $-m\in [T,2T]$
is analogous). Let $h$ be an even positive function satisfying
$h(0)=1$ with Fourier transforms $g$ smooth and supported on
$[-1,1]$. Similar to the previous case, we can bound the function
$h(Rr)\gg \frac{e^{R(1/2-s)}}{R}$ from below, and from positivity
this is also a lower bound for the sum over all representations
$\pi_k\cong\pi_{s_{k,m}}\otimes\frakD_{m}$ occurring in
$L^2(\Gamma\bs G,\rho,(0,m))$,
\[\frac{\exp(R|1/2-s|)}{R}\ll \sum_k h(Rr_{k,m}),\]
where as usual  $s_{k,m}=\frac{1}{2}+r_{k,m}$. Now use the trace
formula with weight $(0,m)$ to transform this sum to a sum over
conjugacy classes
\begin{eqnarray*}
\sum_k h(Rr_{k,m})=\frac{c_1 \chi_\rho(1)(2|m|-1)}{16R^2\pi^2}\int_{\bbR}
h(r)r\tanh(\pi Rr)dr\\
 \frac{1}{R}\sum_{\{\gamma\}\in h.e}c_\gamma
\chi_\rho(\gamma)\frac{g_1(\frac{l_{\gamma_1}}{R})}{\sinh(\frac{l_{\gamma_2}}{2}
)} \frac{ie^{i(2|m|-1)\theta_2}}{2\sin(\theta_2)}\\
\end{eqnarray*}

In order to evaluate this sum we first add the contribution of all
other wights in a window around $T$ (thus only making it bigger).
Let $\psi$ be a positive smooth function supported on
$[\frac{1}{2},\frac{5}{2}]$ with $\psi(1)=1$, and consider the sum
\[\sum_m\psi(\frac{m}{T})\sum_{k} h(Rr_{k,m}).\]
From the positivity of $\psi$ this sum is still bounded from below
by $\frac{\exp(R|1/2-s|)}{R}$. On the other hand if we replace the
inner sum with the right hand side of the trace formula we get
\begin{eqnarray*}
\lefteqn{\sum_m\psi(\frac{m}{T})\sum_k h(Rr_{k,m})=}\\
&&\frac{c_1 \chi_\rho(1)}{16R^2\pi^2}\int_{\bbR} h(r)r\tanh(\pi R r)dr\sum_m
(2|m|-1)\psi(\frac{m}{T})\\
&&+ \frac{1}{R}\sum_{\{\gamma\}\in h.e}c_\gamma
\chi_\rho(\gamma)\frac{g_1(\frac{l_{\gamma_1}}{R})}{\sinh(\frac{l_{\gamma_2}}{2}
)\sin(\theta_2)}\sum_{m} \psi(\frac{m}{T})e^{i(2m-1)\theta_2}\\
\end{eqnarray*}

The first term is bounded by $O(\chi_\rho(1)\frac{T^2}{R^2})$
(recall we are only considering $\frac{T}{2}\leq m\leq
\frac{5T}{2}$). We can bound the second term by
\[\frac{\chi_\rho(1)}{R}\mathop{\sum_{\{\gamma\}\in h.e}}_{l_{\gamma_1}\leq
R}|\frac{c_\gamma}{\sinh(\frac{l_{\gamma_2}}{2})\sin(\theta_2)}||\sum_{m}
\psi(\frac{m}{T})e^{2im\theta_2}|\]
Now use Poisson summation to get
\[|\sum_{m}
\psi(\frac{m}{T})e^{2im\theta}|=|T\sum_{m}\hat{\psi}(T(\theta+2m))|,\]
where $\hat{\psi}$ is the Fourier transform of $\psi$. From the fast
decay of $\hat{\psi}$ we can deduce that the main contribution is
given by $T\hat{\psi}(T\theta)$, which is bounded by
$O_\epsilon(T^{-1})$ for $\theta \geq T^{-1+\epsilon}$ and by $O(T)$
for $\theta\leq T^{-1+\epsilon}$.

Thus, exactly as in the previous case we get
\begin{eqnarray*}
\frac{\exp(R|1/2-s|)}{R}\ll_\epsilon \dim(\rho)(\frac{T^2}{R^2}+
\frac{T}{R}\mathop{\sum_{|t_1|\leq
e^{R/2}}}_{2-\frac{1}{T^{2-\epsilon}}<|t_2|<2}
\frac{F_\Gamma(t)}{\sqrt{(t_1^2-4)(t_2^2-4)}}\\
+\frac{1}{RT}\mathop{\sum_{|t_1|\leq e^{R/2}}}_{|t_2|\leq
2}\frac{F_\Gamma(t)}{\sqrt{(t_1^2-4)(t_2^2-4)}})
\end{eqnarray*}
and setting $R=c\log T$ concludes the proof.
\end{proof}

Theorem \ref{t:gap1} is now reduced to the following two counting
arguments:
\begin{prop}[First counting argument]\label{p:count1}
There is a constant $C$ (depending only on $\Gamma$) such that for
any $(x_1,x_2)\in\bbR^2$
\[\sharp\set{t\in\Tr(\Gamma):|t_1-x_1|\leq T_1,|t_2-x_2|\leq T_2}\leq
1+CT_1T_2\]
\end{prop}

\begin{prop}[Second counting argument]\label{p:count2}
\[\frac{F_\Gamma(t)}{\sqrt{|(t_1^2-4)(t_2^2-4)|}}\ll_\epsilon
|(t_1^2-4)(t_2^2-4)|^{\epsilon}\]
\end{prop}

\begin{proof}[Proof of Theorem \ref{t:gap1}]
We give the proof of Theorem \ref{t:gap1} from the two counting
arguments. Assume that $\pi\cong\pi_1\otimes\pi_2$ occur in
$L^2(\G\bs G,\rho)$ and satisfies the hypothesis of Theorem
\ref{t:gap1}. For fixed $c>0$ and any $\epsilon>0$ let
$\epsilon_1,\epsilon_2$ such that $\epsilon=\epsilon_1+c\epsilon_2$,
then by Proposition \ref{p:geom} we have
\begin{eqnarray*}
T^{c|1/2-s_1|}\ll_\epsilon \dim(\rho)(
\frac{T^2}{\log(T)}+T\!\!\!\!\!\!\!\!\!\mathop{\sum_{|t_1|\leq
T^{c/2}}}_{|t_2|=2+O(T^{-2+\epsilon_1})}\!\!\frac{F_\Gamma(t)}{\sqrt{
|(t_1^2-4)(t_2^2-4)|}}\\
+\frac{1}{T}\!\!\mathop{\sum_{|t_1|\leq T^{c/2}}}_{|t_2|\leq
2}\!\!\frac{F_\Gamma(t)}{\sqrt{|(t_1^2-4)(t_2^2-4)|}}).
\end{eqnarray*}
The second counting argument (Proposition \ref{p:count2}) together
with the bound $|(t_1^2-4)(t_2^2-4)|\ll T^{c}$ (which holds for all
pairs $(t_1,t_2)$ appearing in the sum) gives
\begin{eqnarray*}
\lefteqn{T^{c|1/2-s_1|}\ll_\epsilon \dim(\rho)( \frac{T^2}{\log(T)}}\\
&&+T^{1+c\epsilon_2}\sharp\set{t\in\Tr(\Gamma):|t_1|\leq
T^{c/2},\;|t_2|=2+O(T^{-2+\epsilon_1})}\\
&&+T^{-1+c\epsilon_2}\sharp\set{t\in\Tr(\Gamma):|t_1|\leq
T^{c/2},\;|t_2|\leq 2})
\end{eqnarray*}
Now by the first counting argument (Proposition \ref{p:count2}) we
get
\[T^{c|1/2-s_1|}\ll_\epsilon \dim(\rho)(
\frac{T^2}{\log(T)}+T^{c/2-1+\epsilon})\]
concluding the proof.
\end{proof}

\subsection{Modifications for $d>2$}\label{s:generald}
Let $\G\subset \PSL(2,\bbR)^d$ be a lattice (derived from quaternion algebra)
and $\rho$ a unitary representation.
Assume that $\pi\cong \pi_1\otimes\pi_2\otimes\cdots\otimes\pi_d$ occurs in
$L^2(\G\bs
G,\rho)$ with $\pi_1\cong \pi_{s_1}$ complementary series and let $J_1,J_2,J_3\subset \{2,\ldots,
n\}$ with $J_1$
the set of indices for which $\pi_j$ is either complementary series or principal
series with $r_j< 1$, $J_2$ the set of indices for which $\pi_j$ is of
principal series with $r_j>1$ and $J_3$ the set of
indices for which $\pi_j\cong \frakD_{m_j}$. For $j\in J_2\cup J_3$ let $T_j\geq 1$ be such that $r_j\in [T_j,2T_j]$ for $j\in J_2$ and $|m_j|\in[T_j,2T_j]$ for $j\in J_3$ and let $T=\prod_{j\in J_2\cup J_3}T_j$.
With these notations the statement of Theorem \ref{t:gap1} remains the same, that is for any
$c>0$
\begin{eqnarray}\label{e:geometric}
T^{c|\frac{1}{2}-s_1|}\ll_\epsilon\dim(\rho)(\frac{T^2}{\log(T)}+T^{\frac{c}{2}
-1+\epsilon}).
\end{eqnarray}
Theorem \ref{t:gap} now follows from (\ref{e:geometric}) just as in
the case of $d=2$. In order to prove the asymptotic estimate (\ref{e:geometric}) in this setting, we apply the trace formula (and Poisson summation in the $m_j$ variables) to the test function
\[h(r;m)=h_1(c\log(T)r_1)\prod_{j\in
J_1}h_1(r_j)\prod_{j\in J_2}h_2(\frac{r_j}{T_j})\prod_{j\in J_3}\psi(\frac{m_j}{T_j}),\]
where $h_1,h_2$ and $\psi$ are as in the proof of Proposition \ref{p:geom}.
The result then follows from the same estimates as in the proof of Proposition
\ref{p:geom} (and some elementary combinatorics) together with the natural
generalization of the two counting arguments above (the proofs of the counting arguments given below are for any $d\geq 2$).

\section{Counting solutions}
In the following section we give proofs for the two counting
arguments. Let $\calA$ be a quaternion algebra unramified in $d$
real places, let $\calR$ a maximal order in $\calA$ and let
$\Gamma\subset\PSL(2,\bbR)^d$ be the corresponding lattice.
\subsection{First counting argument}
The proof of the first counting argument is a direct result of the
following estimate on the number of lattice points coming from a
number field lying inside a rectangular box. Let $L/\bbQ$ be a
totally real number field of degree $n$ and $\iota_1,\ldots,\iota_n$
the different embeddings of $L$ to $\bbR$. We then think of
$\calO_L$ as a lattice in $\bbR^n$ via the map $\calO_L\ni t \mapsto
(\iota_1(t),\ldots,\iota_n(t))\in\bbR^n$. We show that the number of
such lattice points in any box parallel to the axes is bounded by
the volume of the box.
\begin{rem}
Note that if the volume of the box is large with comparison to the
area of its boundary, then this result would follow from the fact
that the volume of the fundamental domain of this lattice is given
by the square root of the discriminant and is hence $>1$. However,
we are interested in particular in the case where the box is narrow
in one direction and long in the other so that this type of argument
will not work. Fortunately, there is a simple argument that works
uniformly for all such boxes.
\end{rem}

\begin{lem}\label{l:count}
For any $B\subset \bbR^n$ a box parallel to the axes the number of
lattice points in this box satisfy $|B\cap \calO_L|\leq 1+\vol(B).$
\end{lem}
\begin{proof}
The only thing we will use is that for any $0\neq t\in \calO_L$ we
have $N_{L/\bbQ}(t)\in\bbZ\setminus \{0\}$, and hence
$N_{L/\bbQ}(t)>1$. Let $T_1,\ldots,T_n>0$ and $\vec{x}\in\bbR^n$
such that
\[B=\set{t\in \bbR^n:|t_j-x_j|\leq T_j}.\]
Now, decompose the segment $[x_1-T_1,x_1+T_1]$ into short segments
of length $\frac{1}{cT_2\cdots T_n}$ with $c>2^{n-1}$. Then there
are less then $2cT_1\cdots T_n+1$ segments (one of them might be
shorter). Now, if there were more then $2cT_1\cdots T_n+1$ elements
in $B\cap \calO_L$, then there must be at least two elements $t\neq
t'$ such that $\iota_1(t),\iota_1(t')$ lie in the same segment.
Consequently, we get that $|\iota_1(t-t')|<\frac{1}{cT_2\cdots
T_n}$, and on the other hand for $j\neq 1$, $|\iota_j(t-t')|\leq
2T_j$. We thus get that $|N_{L/\bbQ}(t-t')|\leq\frac{2^{n-1}}{c}<1$
in contradiction. We have thus shown that $|B\cap \calO_L|<
2cT_1\cdots T_n+1$ for any $c>2^{n-1}$ implying that indeed
\[|B\cap \calO_L|\leq  2^nT_1\cdots T_n +1=\vol(B)+1.\]
\end{proof}

\begin{proof}[Proof of Proposition \ref{p:count1}]
Let $\Gamma\subset \PSL(2,\bbR)^d$ be a lattice derived from a quaternion algebra over a
totally real number field $L$. Denote by $\iota_1,\ldots\iota_n$ the
different embeddings of $L$ into $\bbR$. Let
$(t_1,t_2,\ldots,t_d)=\Tr(\gamma)\in\Tr(\Gamma)$. Then there is
$\alpha\in\calR^1$ such that $\gamma_j=\iota_j(\alpha)$ for $1\leq j\leq d$.
Let $t=\Tr_\calA(\alpha)\in\calO_L$ then $t_j=\iota_j(t)$ for $j\leq d$ and for $j>d$ we have
$\iota_j(\calR^1)\subseteq SO(2)$ so $|\iota_j(t)|\leq 2$.
Consequently, we can bound
\[\sharp\set{(t_1,\ldots,t_d)\in\Tr(\Gamma):\forall j\leq d,\;|t_j-x_j|\leq T_j},\]
by the number of elements in
\[\set{t\in\calO_L:\forall j\leq d,\;|t_j-x_j|\leq T_j \mbox{ and }\forall j>d,\; |t_j|\leq
2}\]
which is bounded by $1+2^{2n-d}T_1T_2\cdots T_d$.
\end{proof}

\subsection{Arithmetic formula}
Before we proceed with the proof of the second counting argument, we
give a formula for the counting function $F_\Gamma(t)$ in terms of
certain arithmetic invariants (see appendix \ref{s:algebraic} for
the related background from algebraic number theory).

Let $\alpha\in \calR^1$ not in the center, and denote
$\Tr_\calA(\alpha)=a\in \calO_L$ and $D=a^2-4$. The centralizer
$\calA_\alpha=\set{\beta\in
\calA|\beta\alpha=\alpha\beta}=L(\alpha)$ is a quadratic field
extension isomorphic to $L(\sqrt{D})$ (via the map $\alpha\mapsto
\frac{a+\sqrt{D}}{2}$). Let $\g=\g_\alpha\subset L$ be the set
$$\g=\set{u\in L|\exists x\in L,\;x+u\alpha\in \calR}.$$
\begin{lem}\label{l:divisor}
The set $\g$ is a fractional ideal containing $\calO_L$ (i.e.,
$\g^{-1}$ is an integral ideal). The ideal
$d=d_\alpha=\g_\alpha^{2}D\subset \calO_L$ is also an integral ideal.
\end{lem}
\begin{proof}
The first assertion is obvious. For the second part we show that any
$u\in \g$ satisfies $u^2D\in \calO_L$. Indeed, for any $u \in \g$
there is $\beta=x+u\alpha\in \calR$. Since we know that
$\calN_\calA(\beta)=x^2+u^2+xua\in \calO_L$ and
$\Tr_A(\beta)=2x+ua\in \calO_L$, we can deduce that
$$u^2D=(2x+ua)^2-4(x^2+xua+u^2)=\Tr_{\calA}(\beta)^2-4\calN_A(\beta)\in
\calO_L.$$
\end{proof}

For $D,d$ as above let $K=L(\sqrt{D})$ and denote by $\calO_K$ the
integers of $K$. Define the ring
 \[\calO_{D,d}=\set{\frac{t+u\sqrt{D}}{2}\in \calO_K : d|(u^2D)}.\]
This is an order inside $\calO_K$ \cite[Proposition 5.5]{Efrat87}
and its relative discriminant over $L$ is precisely the ideal $d$
(see Lemma \ref{l:disc}).
\begin{prop}\label{p:cor}
Let $\alpha\in \calR^1$, let $D=\Tr_{\calA}(\alpha)^2-4$ and
$d=d_\alpha\subset \calO_L$ as above. Under the map $\alpha\to
\frac{1+\sqrt{D}}{2}$, the order $\calR_\alpha=\calA_\alpha\cap
\calR$ is mapped onto $\calO_{D,d}$.
\end{prop}
\begin{proof}
Denote by $\calO_\alpha$ the image of $\calA_\alpha\cap \calR$ under
this map, so
\[\calO_\alpha=\set{\frac{t+u\sqrt{D}}{2}\in
L(\sqrt{D})|\frac{t+u(2\alpha-a)}{2}\in \calR}.\]
The condition $\frac{t+u(2\alpha-a)}{2}\in \calR$ implies that
$t=\Tr_{\calA}(\frac{t+u(2\alpha-a)}{2})\in \calO_L$ and that $u\in
\g$. Note that for any $u\in L$ we have the equivalence
$d|(u^2D)\Leftrightarrow \g^2D|(u^2)(D)\Leftrightarrow
\g^2|(u)^2\Leftrightarrow u\in \g$. Hence $\calO_\alpha\subset
\calO_{D,d}$.

For the other direction let $\frac{t+u\sqrt{D}}{2}\in \calO_{D,d}$.
In particular $u\in \g$ and hence there is $\beta=x+u\alpha\in
\calR$. Let $\tilde{t}=\Tr_{\calA}(\beta)\in \calO_L$ then
$\beta=\frac{\tilde{t}-ua}{2}+u\alpha$ and hence
$4\calN_A(\beta)=\tilde{t}^2-u^2D\in 4\calO_L$. But from the
definition of $\calO_{D,d}$ we also know $t^2-u^2D\in4\calO_L$,
hence $t^2-\tilde{t}^2\in 4\calO_L$ and $t\equiv
\tilde{t}\pmod{2\calO_L}$. Now
$\frac{t+u(2\alpha-a)}{2}-\frac{\tilde{t}+u(2\alpha-a)}{2}=\frac{t-\tilde{t}}{2}
\in
\calO_L\subset \calR$, and hence $\frac{t+u(2\alpha-a)}{2}\in \calR$
and $\frac{t+u\sqrt{D}}{2}\in \calO_\alpha$.
\end{proof}

\begin{prop}\label{p:rank}
With the above notation assume that $\iota_j(D)\in\bbR$ is positive
for $j=1,\ldots, m_0$ and negative for $j=m_0+1,\ldots, n$ for some
$1\leq m_0\leq m$. Then $\calO_{D,d}^1$ is a free group of rank
$m_0$.
\end{prop}
\begin{proof}
See \cite[proof of Theorem 5.7]{Efrat87}.
\end{proof}

\begin{defn}
Let $\epsilon_1,\epsilon_2,\ldots,\epsilon_{m_0}$ be generators for
$\calO_{D,d}^1$. For each $j=1,\ldots, m_0$ choose one place of
$L(\sqrt{D})$ above $\iota_j$ (that we also denote by $\iota_j$).
Define the regulator $\reg(\calO_{D,d}^1)$ as the absolute value of
the determinant of the $m_0\times m_0$ matrix given by
$a_{i,j}=\log|\iota_i(\epsilon_j)|$.
\end{defn}

\begin{prop}\label{p:vol}
Let $\alpha\in\calR^1$ and denote by $D=\Tr_{\calA}(\alpha)^2-4$,
and $d=\g_\alpha^2D$ as above. Then $\vol(\Gamma_\gamma\bs
G_\gamma)=\reg(\calO_{D,d}^1)$ where $\gamma=\iota(\alpha)$.
\end{prop}
\begin{proof}
See \cite[Proposition 6.1]{Efrat87}
\end{proof}

\begin{prop}\label{p:form}
For $\Gamma$ as above and $t\in \calO_L$
\[F_\Gamma(t)= \sum_{d|(D)} \reg(\calO_{D,d}^1)l(\calO_{D,d}).\]
where the sum is over all ideals $d$ such that $\frac{(D)}{d}$ is a
square of an integral ideal, and $l(\calO_{D,d})$ is the number of
conjugacy classes of centralizers corresponding to $\calO_{D,d}$.
\end{prop}
\begin{proof}
Recall that
\[F_\Gamma(t)=\mathop{\sum_{\{\gamma\}\in\Gamma^\sharp}}_{\Tr(\gamma)=t}
\vol(\Gamma_\gamma\bs G_\gamma).\]
We can assume that $t_j=\iota_j(t)$ for some
$t\in\calO_L$ and think on $F_\Gamma$ as a function on $\calO_L$.
Replace the sum over conjugacy classes $\{\gamma\}\in\Gamma^\sharp$
to a sum over conjugacy classes $\{\alpha\}\in {\calR^1}^\sharp$.
Next for $\gamma=\iota(\alpha)$, by proposition \ref{p:vol}, we have
that $\vol(\Gamma_\gamma\bs G_\gamma)=\reg(\calO_{D,d}^1)$ where
$D=t^2-4$ and $d|(D)$ is the ideal corresponding to $\alpha$ as in
proposition \ref{p:cor}. Consequently we can write
\[F_\Gamma(t)=\sum_{d|(D)}
\reg(\calO_{D,d}^1)\sharp\set{\{\alpha\}|\Tr_{\calA}(\alpha)=t,\; d_\alpha=d}\]
where the sum is over all integral ideals $d|(D)$ such that $(D)/d$
is a square of an integral ideal. Now consider the map sending each
conjugacy class $\{\alpha\}$ to the conjugacy class of its
centralizer $\{\calR_\alpha^1\}$. Note that two different elements
of $\Gamma$ with the same trace do not commute \cite[Lemma
7.4]{Efrat87}, hence this map is a bijection of the set
$$\set{\{\alpha\}|\Tr_{\calA}(\alpha)=t,\; d_\alpha=d},$$
and the set of conjugacy classes of centralizers corresponding to $\calO_{D,d}$.
Consequently
we have $\sharp\set{\{\alpha\}|\Tr_{\calA}(\alpha)=t,\;
d_\alpha=d}=l(\calO_{D,d})$.
\end{proof}

\subsection{Second counting argument}
Fix $\alpha\in \calR^1$ (not in the center), let $K=L(\alpha)$ be
the corresponding quadratic extension and let $\calO=\calR\cap K$.
Then by Proposition \ref{p:cor} we have $\calO\cong \calO_{D,d}$
where $D=\Tr_{\calA}(\alpha)^2-4$ and $d=d_\alpha$ as in Lemma
\ref{l:divisor}. Note that if $\alpha'\in \calR^1$ is conjugate (in
$\calR^1$) to $\alpha$, then $D'=D$ and $d'=d$, so the corresponding
rings are also the same. Recall that $l(\calO_{D,d})$ is the number of
$\calR^1$-conjugacy classes of centralizers that correspond to
$\calO_{D,d}$. In the notation of Eichler (see
\cite{Eichler56,Shimizu}) this is the number of $\calR^1$-conjugacy
classes of optimal embeddings of $\calO$ into the maximal order
$\calR$. We now wish to give an upper bound for this number, or
rather to the product $l(\calO_{D,d})\reg(\calO_{D,d}^1)$.

Let $\calC(\calO)$ denote the class group (or the Picard group) of
$\calO$ and denote by $\sharp\calC(\calO)=h(\calO)$ the class
number. Let $H$ denote the group of two sided ideals of $\calR$ and
$H'$ denote the subgroup of all ideals generated by $\calO$-ideals.
Then $[H:H']=\prod_{\calP|\mathfrak{d}}(1-(\frac{\calO}{\calP}))$
where $(\frac{\calO}{\calP})$ stands for Artin's symbol and
$\mathfrak{d}$ denotes the discriminant of $\calA$ over $L$
\cite[equation 47]{Shimizu}. In particular, $[H:H']$ is bounded by a
constant $c(\mathfrak{d})$ depending only on $\mathfrak{d}$.

\begin{prop}\label{p:ClassForm}
\[l(\calO)\leq C_1 \frac{h(\calO)}{[\calO^*:\calO^1\calO_L^*]}.\]
where $C_1=c(\mathfrak{d})[\calR^*:\calR^1\calO_L^*]$ is a constant
depending only on the quaternion algebra.
\end{prop}
\begin{proof}
Let $\kappa$ be the number of pairs $(\mathfrak M,\mathfrak a)\in
H/H'\times \calC(\calO)$ such that the ideal $\mathfrak M \mathfrak
a=\calR\mu$ is principal. We then have  \cite[equation
45]{Shimizu}\footnote{In \cite{Shimizu} it is stated for $K/L$ an
imaginary extension, but the same proof holds here without
changes.},
\[l(\calO)=\frac{[\calR^*:\calR^1\calO_L^*]}{2[\calO^*:\calO^1\calO_L^*]}
\kappa.\]
 Now use the bound $\kappa\leq [H:H']h(\calO)\leq c(\mathfrak{d})h(\calO)$ to
conclude the proof.
\end{proof}

\begin{prop}\label{p:bound1}
\[l(\calO_{D,d})\reg(\calO_{D,d}^1)\leq C_2
\sqrt{N_{L/\bbQ}(d)}\res_{s=1}\zeta_K(s)\]
where $\zeta_K(s)$ is the Dedekind Zeta function corresponding to
$K$ and $C_2$ is a constant depending only on the quaternion
algebra.
\end{prop}
\begin{proof}
Denote by $\reg(\calO_{D,d}^*)$ and $\reg(\calO_L^*)$ the regulators
of $\calO_{D,d}^*$ and $\calO_L^*$ respectively. Combining the bound
on $l(\calO_{D,d})$ (Proposition \ref{p:ClassForm}), and the
relation
$\reg(\calO_{D,d}^*)=\frac{\reg(\calO_{D,d}^1)\reg(\calO_L^*)}{[\calO_{D,d}
^*:\calO_{D,d}^1\calO_L^*]}$
(Proposition \ref{p:reg}) we get
\[l(O_{D,d})\reg(O_{D,d}^1)\leq C_1
\frac{h(\calO_{D,d})\reg(\calO_{D,d}^*)}{\reg(\calO_L^*)}.\]
with $C_1$ the constant in Proposition \ref{p:ClassForm}.

Let $D_{K/L}\subseteq \calO_L$ denote the relative discriminant of
$K/L$, let $\f=\set{x\in\calO_K| xO_K\subseteq \calO_{D,d}}$ denote
the conductor of $\calO_{D,d}$ and let $\f_0=\f\cap \calO_L$. We can
bound (see Corollary \ref{c:classes})
\[h(\calO_{D,d})\reg(\calO_{D,d}^*)\leq
2^{n+1}N_{L/\bbQ}(\f_0)h(\calO_K)\reg(\calO_K^*).\]
Now use the class number formula (see e.g.,  \cite[Corollary
5.11]{Neukirch99})
\[h(\calO_K)\reg(\calO_K^*)=\frac{2\sqrt{D_K}}{2^{n+m}\pi^{n-m}}\res_{s=1}
\zeta_K(s),\]
 to get that
\[l(\calO_{D,d})\reg(\calO_{D,d}^1)\leq \frac{C_1}{\reg(O_L^*)}
N_{L/\bbQ}(\f_0)\sqrt{D_K}\res_{s=1}\zeta_K(s).\]
Finally replace $D_K=\frac{N_{L/\bbQ}(D_{K/L})}{D_L^2}$ (Proposition
\ref{p:disc}) and $D_{K/L}f_0^2=d$ (Proposition \ref{p:disc2}) to
conclude that
\[l(\calO_{D,d})\reg(\calO_{D,d}^1)\leq \frac{C_1}{\reg(\calO_L^*)D_L}
\sqrt{N_{L/\bbQ}(d)}\res_{s=1}\zeta_K(s).\]
\end{proof}

\begin{proof}[Proof of Proposition \ref{p:count2}]
By Propositions \ref{p:form} and \ref{p:bound1} we get
\[F_\Gamma(t)= \sum_{d|(D)} \reg(\calO_{D,d}^1)l(\calO_{D,d})\ll
\sum_{d|(D)}\sqrt{N_{L/\bbQ}(d)}\res_{s=1}\zeta_K(s)\]
with $D=t^2-4$ and $K=L(\sqrt{D})$. For any $d|(D)$ we can bound
$N_{L/\bbQ}(d)\leq N_{L/\bbQ}(D)$ so
\[F_\Gamma(t)\ll \sqrt{N_{L/Q}(D)}\res_{s=1}\zeta_K(s)\sharp\set{a\subset
\calO_K|a^2|(D)},\]
The number of ideal divisors of $(D)$ is bounded by
$O(N_{L/Q}(D)^\epsilon)$ and for the residue of the Zeta function we
have \cite[Theorem 1]{Louboutin}
$$\res_{s=1}\zeta_K(s)\leq (\log(D_K))^{2n+1}.$$
Since $D_K\leq N_{L/\bbQ}(D)\ll\prod_j|(t_j^2-4)|$ indeed
\[F_\Gamma(t)\ll_\epsilon N_{L/Q}(D)^{1/2+\epsilon}\ll
\prod_j|(t_j^2-4)|^{1/2+\epsilon}.\]
\end{proof}

\section{Application for Selbergs Zeta function}
We conclude with the proof of Corollary \ref{c:zeta} from Theorem
\ref{t:gap} giving a zero free region for the Selberg Zeta functions
$Z_m(s,\G)$. For $\Gamma\subset \PSL(2,\bbR)^2$ irreducible without
torsion and any $m\geq 1$ the corresponding Zeta function is given
by
\begin{equation}
Z_m(s,\Gamma)=\prod_{\{\gamma\}^*_\Gamma}\mathop{\prod_{\nu=0}^\infty}_{|i|<
m }(1-\epsilon_\gamma^iN(\gamma)^{-s-\nu})^{-1}
\end{equation}
where the product is over all primitive conjugacy classes in $\Gamma$ that
are hyperbolic in the first coordinate and
elliptic in the second. Using the trace formula with wight $(0,m)$
(as in section \ref{s:trace}) one shows that $Z_m(s,\Gamma)$ is
entire (except when $m=1$ where it has a simple pole at $s=1$) and
satisfies a functional equation relating $s$ and $1-s$. Also
$Z_m(s,\G)$ has trivial zeros at the negative integers and spectral
zeros at the points $s_k$ such that $\pi_{s_k}\otimes\frakD_m$
appears in the decomposition of $L^2(\G\bs G)$. Now Theorem
\ref{t:gap} implies that for any $t_0>5/6$ there are only finitely
many $\pi_k=\pi_{s_{k}}\otimes \frakD_m$ in the decomposition of
$L^2(\G\bs G)$ with $1/2<s_{k,1}<t_0$. In particular for
sufficiently large $m_0$ the half plane $\Re(s)>t_0$ is a zero free
region for all the $Z_m(s,\G)$'s with $m>m_0$.

\appendix
\section{Algebraic background}\label{s:algebraic}
In this appendix we provide some background and collect a number of
results from algebraic number theory that we have used. The main
reference for this section is \cite{Neukirch99}.
\subsection{Discriminants}
Let $K/L$ be an extension of number fields and let $\calO_K$ and
$\calO_L$ denote the corresponding rings of integers. For any basis
$\{x_j\}$ of $K/L$ the discriminant of the basis is defined as the
determinant of the matrix $\Tr_{K/L}(x_ix_j)$. An order
$\calO\subseteq\calO_K$ is a subring that has rank $[K:\bbQ]$ as a
$\bbZ$ module. For any order $\calO\subseteq\calO_K$, the relative
discriminant $d=d(\calO/\calO_L)$ is the ideal in $\calO_L$
generated by the discriminants of all bases for $K/L$ that lie in
$\calO$. When $\calO=\calO_K$ is the full ring of integers we denote
$d(\calO_K/\calO_L)=D_{K/L}$ the relative discriminant of $K/L$. The
relative discriminant of $K/\bbQ$ (respectively $L/\bbQ$) is a
principal ideal in $\bbZ$, the generator of this ideal denoted by
$D_K$ (respectively $D_L$) is the discriminant of the field. We then
have the following relation:
\begin{prop} \label{p:disc}
\[D_K=D_L^{[K:L]}N_{L/\bbQ}(D_{K/L})\]
\end{prop}
\begin{proof}
 \cite[Corollary 2.10]{Neukirch99}
\end{proof}

Assume now that $K/L$ is a quadratic extension. Let
$\calO\subseteq\calO_K$ be an order. The conductor of $\calO$ is
defined by
$$\f=\f(\calO)=\set{x\in\calO_K|x\calO_K\subset \calO}.$$
This is an ideal in $\calO_K$ that measures how far is the order
$\calO$ from the full ring of integers. Denote by
$\f_0=\f\cap\calO_L$ the ideal in $\calO_L$ lying under it. We then
have the following:
\begin{prop}\label{p:disc2}
Let $d=d(\calO/\calO_L)$ and $D_{K/L}=d(\calO_K/\calO_L)$ denote the
relative discriminants of $\calO$ and $\calO_K$ over $\calO_L$
respectively. Then $d=\f_0^2D_{K/L}$.
\end{prop}

This result is well known, however since we could not find a good
reference we will include a short proof. We first need a few
preparations. Fix $D\in \calO_L$ such that $K=L(\sqrt{D})$ and
define
$$\g=\g(\calO,D)=\set{u\in L|\exists t\in\calO_L\; \mathrm{s.t.}\;
\frac{t+u\sqrt{D}}{2}\in\calO}.$$
Then $\g$ is a fractional ideal containing $\calO_L$ and $\g^2D$ is
an integral ideal. (Note that the ideal $\g$ depends on the choice
of $D$ but the product $\g^2D$ does not).
\begin{lem}\label{l:J}
We have $\calO=\set{\frac{t+u\sqrt{D}}{2}\in\calO_K|u\in\g}$.
\end{lem}
\begin{proof}
The inclusion
$\calO\subseteq\set{\frac{t+u\sqrt{D}}{2}\in\calO_K|u\in\g}$ is
obvious. To show the other direction assume that
$\beta_1=\frac{t_1+u\sqrt{D}}{2}\in\calO_K$ with $u\in\g$. Then
there is $t_2\in\calO_L$ such that
$\beta_2=\frac{t_2+u\sqrt{D}}{2}$. For both $j=1,2$ we have that
$N_{K/L}(\beta_j)=\frac{t_j^2-u^2D}{4}\in\calO_L$. In particular
$N_{K/L}(\beta_1)-N_{K/L}(\beta_2)=\frac{t_1^2-t_2^2}{4}\in\calO_L$
so $t_1\equiv t_2\pmod{2\calO_L}$. Now since the difference
$\beta_1-\beta_2=\frac{t_1-t_2}{2}\in\calO_L\subseteq\calO$ and
$\beta_2\in\calO$ then $\beta_1\in\calO$ as well.
\end{proof}

\begin{lem}\label{l:disc}
The relative discriminant $d=d(\calO)$ of $\calO$ over $L$ satisfies
$d=\g^2D$.
\end{lem}
\begin{proof}
By definition, the relative discriminant of $\calO$ over $L$ is the
ideal generated by all elements of the form
$\frac{(t_1u_2-t_2u_1)^2D}{4}$ with
$x_j=\frac{t_j+u_j\sqrt{D}}{2}\in\calO$ such that $(x_1,x_2)$ is a
basis for $K/L$. Now notice that
\[\g=\set{\frac{(t_1u_2-t_2u_1)}{2}|x_j=\frac{t_j+u_j\sqrt{D}}{2}\in\calO}.\]
To see this, note that if $x_1,x_2\in\calO$ then the product
$x_1x_2\in\calO$ as well implying that
$\frac{(t_1u_2-t_2u_1)}{2}\in\g$. For the other direction for any
$u\in\g$ take $x_1=\frac{t+u\sqrt{D}}{2}\in\calO$ and
$x_2=1=\frac{2+0\sqrt{D}}{2}$. Consequently, the discriminant is the
ideal generated by $\set{x^2D|x\in \g}$ which is precisely $\g^2D$.
\end{proof}

\begin{proof}[Proof of Proposition \ref{p:disc2}]
Fix $D\in\calO_L$ such that $K=L(\sqrt{D})$. Let
$\g_1=\g(\calO_K,D)$ and $\g_2=\g(\calO,D)$. Then by Lemma
\ref{l:disc} we have that $D_{K/L}=D\g_1^2$ and $d=D\g_2^2$. It thus
remains to show that $\f_0=\g_1^{-1}\g_2$. Now, by definition
$\f_0=\set{x\in\calO_L|\forall \beta\in\calO_K,\;x\beta\in\calO},$
and by Lemma \ref{l:J} this is the same as $\set{x\in\calO_L|\forall
u\in\g_1,\;xu\in\g_2}=\g_1^{-1}\g_2$.
\end{proof}

\subsection{Regulators}\label{s:reg}
Let $K$ be number field with $r_1$ real places and $r_2$ (conjugate
pairs of) complex places. For any place $\nu$ of $K$ define the
corresponding norm by $\norm{x}_\nu=|\nu(x)|$ when $\nu$ is real and
$\norm{x}_\nu=|\nu(x)|^2$ when $\nu$ is complex. Let $\calO_K$
denote the ring of integers in $K$. By Dirichlet unit theorem the
group of units $\calO_K^*$ is a free group of rank $r=r_1+r_2-1$.
The regulator $\reg(\calO_K^*)$ of this group is the absolute value
of the determinant of the matrix
$a_{i,j}=(\log\norm{\epsilon_i}_{\nu_j})$ where $\nu_j$ goes over
$r$ out of the $r+1$ places and $\epsilon_1,\ldots,\epsilon_r$ are
generators for the group of units (this is independent of the choice
of generators or places). There is a geometric interpretation of the
regulator. Consider the logarithmic map from $\calO_K^*$ to
$\bbR^{d+1}$ sending
$$\epsilon\mapsto(\log(\norm{\epsilon}_{\nu_1}),\ldots,\log(\norm{\epsilon}_{
\nu_{r+1}})).$$
Then the image of $\calO_K$ is a lattice of rank $r$ inside
$\bbR^{r+1}$ with co-volume given by $\sqrt{r+1}\reg{\calO_K^*}$. If
$U\subset\calO_K^*$ is a subgroup of finite index, then it is also
of the same rank and we can define the regulator of $U$ in the same
way (by taking $\epsilon_1,\ldots,\epsilon_r$ to be generators for
$U$). We then have the relation $\reg(U)=[O_K^*:U]\reg(\calO_K^*)$
(one can see this by comparing the co-volumes of the corresponding
lattices).

Let $K/L$ be a quadratic extension of a totally real number field
$L$. Let $\calO$ be an order in $\calO_K$, then $\calO^*$ is a
subgroup of $\calO_K^*$ of finite index. Denote by
$\calO^1=\set{x\in\calO|N_{K/L}(x)=1}$ the group of (relative) norm
one elements in $\calO^*$. Let $[L:\bbQ]=n$ and let $1\leq m\leq n$
such that $K$ has $2m$ real places and $(n-m)$ pairs of complex
places. Then $\calO_K^1$ is a free group of rank $m$. Let
$\epsilon_1,\ldots,\epsilon_m$ be generators for $\calO^1$ and
consider the $m\times m$ matrix given by $\log(\nu_i(\epsilon_j))$
where $\nu_i$ goes over the real places of $K$, where from every
pair lying above the same place of $L$ we take only one place. The
regulator $\reg(\calO^1)$ is defined as the absolute value of the
determinant of this matrix.

Note that $\calO^1\calO_L^*$ is a subgroup of $\calO^*$ of finite
index. We have the following relation
\begin{prop}\label{p:reg}
\[\reg(\calO^*)=\frac{\reg(\calO^1\calO_L^*)}{[\calO^*:\calO^1\calO_L^*]}=
\frac{\reg(\calO^1)\reg(\calO_L^*)}{[\calO^*:\calO^1\calO_L^*]}.\]
\end{prop}
\begin{proof}
See \cite[proof of Theorem 1]{CostaFriedman91}
\end{proof}

\subsection{Class numbers}
For a number field $K$ the Class group, $\calC(\calO_K)$, is the
quotient of the group of all fractional ideals of $\calO_K$ with the
subgroup of principal ideals. This is a finite group and its order
$h(\calO_K)$ is the class number of $K$. The Class number formula
relates the class number with other algebraic invariants of the
number field.
\begin{prop}\label{t:class}
\[h(\calO_K)=\frac{w
\sqrt{D_K}}{2^{r_1+r_2}\pi^{r_2}\reg(\calO_K)}\res_{s=1}\zeta_K(s),\]
where $w$ is the number of roots of unity contained in $K$, $D_K$ is
the absolute discriminant, $\zeta_K(s)$ is the Dedekind zeta
function, the numbers $r_1,r_2$ are the number of real and complex
embeddings of $K$, and $\reg(\calO_K^*)$ is the regulator of
$\calO_K$.
\end{prop}
\begin{proof}
See e.g., \cite[Corollary 5.11]{Neukirch99}
\end{proof}

For an order $\calO\subseteq\calO_K$, the fractional ideals do not
necessarily form a group (since not all ideals are invertible).
However one can consider the group of all invertible ideals in
$\calO$. The Picard group of $\calO$ is then the quotient of the
group of all invertible fractional ideals of $\calO$ with the
subgroup of principal ideals. This group is also finite and its
order $h(\calO)$ is called the class number of $\calO$.

The class numbers $h(\calO)$ and $h(\calO_K)$ are related by the
following formula \cite[Theorem 12.12]{Neukirch99}
\begin{equation}\label{e:classes}
h(\calO)=h(\calO_K)\frac{[\calO_K/\f\calO_K)^*:(\calO/\f\calO)^*]}{[
\calO_K^*:\calO^*]}.
\end{equation}
where $\f\subset\calO_K$ is the conductor of $\calO$.
 If we consider the product of the class number and the regulator we get
\begin{prop}
\[h(\calO)\reg(\calO^*)=[\calO_K/\f\calO_K)^*:(\calO/\f\calO)^*]
h(\calO_K)\reg(\calO_K^*).\]
\end{prop}
\begin{proof}
Use the above formula together with the relation
$\reg(\calO^*)=[\calO_K^*:\calO^*]\reg(\calO_K^*)$.
\end{proof}

In the previous setup with $K/L$ a quadratic extension and
$\f_0=\f\cap\calO_L$ this leads to the following bound:
\begin{cor}\label{c:classes}
 \[h(\calO)\reg(\calO^*)\leq 2^{n+1}N_{L/\bbQ}(\f_0)h(\calO_K)\reg(\calO_K^*)\]
\end{cor}
\begin{proof}
We need to give a bound for
$[(\calO_K/\f\calO_K)^*:(\calO/\f\calO)^*]$. Consider the norm map
from $\calN_{K/L}:(\calO/\f\calO)^*\to (\calO_L/\f_0)^*$. The image
of this map contains all the squares in $(\calO_L/\f_0)^*$ which is
a subgroup of index bounded by $2^{n+1}$. We can thus bound from
below $\sharp(\calO/\f\calO)^*\geq 2^{-n-1}\sharp(\calO_L/\f_0)^*$.
Consequently,
\[[(\calO_K/\f\calO_K)^*:(\calO/\f\calO)^*]\leq
2^{n+1}\frac{|(\calO_K/\f\calO_K)^*|}{|(\calO_L/\f\calO_L)^*|}=2^{n+1}\frac{N_{
K/\bbQ}(\f)}{N_{L/\bbQ}(\f_0)}.\]
Write $N_{K/\bbQ}(\f)=N_{L/\bbQ}(N_{K/L}(\f))$ and note that
$\f_0^2\subset N_{K/L}(\f)\subset \f_0$ to get
\[[(\calO_K/\f\calO_K)^*:(\calO/\f\calO)^*]\leq 2^{n+1}N_{L/\bbQ}(\f_0).\]
\end{proof}

%*********************************************************************
\def\cprime{$'$}
\providecommand{\bysame}{\leavevmode\hbox
to3em{\hrulefill}\thinspace}
\providecommand{\MR}{\relax\ifhmode\unskip\space\fi MR }
\providecommand{\MRhref}[2]{%
  \href{http://www.ams.org/mathscinet-getitem?mr=#1}{#2}
} \providecommand{\href}[2]{#2}

\end{document}